\documentclass{cmslatex}
\usepackage[paperwidth=7in, paperheight=10in, margin=.875in]{geometry}
 \usepackage[backref,colorlinks,linkcolor=red,anchorcolor=green,citecolor=blue]{hyperref}
\usepackage{amsfonts,amssymb}
\usepackage{amsmath}
\usepackage{graphicx}
\usepackage{qtree}
\usepackage[linguistics]{forest}
\usepackage{cite}
\usepackage{bm}
\usepackage{enumerate}
\renewcommand{\theequation}{\arabic{section}.\arabic{equation}}

   \allowdisplaybreaks
\begin{document}
 \title{Semilinear Eigenvalue problem: Parametric Analyticity and the Uncertainty Quantification\thanks{Received date, and accepted date (The correct dates will be entered by the editor).}}


          \author{Byeong-Ho Bahn\thanks{University of Massachusetts Amherst, (bban@umass.edu).}
         }

         \pagestyle{myheadings} \markboth{Parametric semilinear eigenvalue problem}{Byeong-Ho Bahn} \maketitle

          \begin{abstract}
In this paper, to the best of our knowledge, we make the first attempt at studying the parametric semilinear elliptic eigenvalue problems with the parametric coefficient and some power-type nonlinearities. The parametric coefficient is assumed to have an affine dependence on the countably many parameters with an appropriate class of sequences of functions. In this paper, we obtain the upper bound estimation for the mixed derivatives of the ground eigenpairs that has the same form obtained recently for the linear eigenvalue problem. The three most essential ingredients for this estimation are the parametric analyticity of the ground eigenpairs, the uniform boundedness of the ground eigenvalue, and the uniform positive gap between ground eigenvalues of related linear operators. All these three ingredients need new techniques and a careful investigation of the nonlinear eigenvalue problem that will be presented in this paper. As an application, considering each parameter as a uniformly distributed random variable, we estimate the expectation of the eigenpairs using a randomly shifted quasi-Monte Carlo lattice rule and show the dimension-independent error bound. 
\end{abstract}


\section{Introduction}\label{intro}
           
\subsection{Main problem.} 
\label{main problem}
 In this paper, we are interested in the smallest eigenvalue, the corresponding eigenfunction, and the corresponding energy  of the parametric semilinear elliptic eigenvalue problem:
\begin{align}
\begin{cases}
\mathcal{S}(u)(x,\bm{y})=\lambda(\bm{y})u(x,\bm{y}), &(x,\bm{y})\in \Omega \times U,\\
u(x,\bm{y})=0, &(x,\bm{y})\in \partial \Omega \times U,
\end{cases} \label{gpe}
\end{align}
where
\begin{align}
\mathcal{S}(u)(x,\bm{y})=-\nabla\cdot\left(a(x)\nabla\right) u(x,\bm{y}) + b(x,\bm{y}) u(x,\bm{y}) +\eta|u(x,\bm{y})|^{p-1}u(x,\bm{y}).
\end{align}
In this problem we assume that $\Omega\subset \mathbb{R}^d$ ($d=1,2,3$) is a bounded open domain with $C^2$ boundary, and $\eta>0$ is a constant. The set $U=\left[-\frac{1}{2},\frac{1}{2}  \right]^\mathbb{N}$ is where our parameters come from. With $(b_i)_{i=1}^\infty\subset L^\infty(\Omega)$, we define $b(x,\bm{y})=b_0(x)+\sum_{i=1}^{\infty} y_i b_i(x)$ which has an affine dependence on $\bm{y}=(y_j)_{j\geq 1}\in U$. Here, we assume that {$a\in C^1(\Omega)$} and there exist $a_{min}>0$ such that $a(x)\geq a_{min}$. We further assume that $b(x,\bm{y})\geq 0$ for all $(x,\bm{y})\in \Omega\times U$, and $(\|b_j\|_{L^\infty})_{j\geq 1}\in \ell^1(\mathbb{N})$. Our choice of $p$ would depend on the dimension $d$ affected by the Sobolev embedding. Specifically, we will assume that $(d,p)\in \mathcal{A}$ where $\mathcal{A}$ is defined in \eqref{dp}.

\subsection{Motivation and the goal.}
\label{goal} For the last few decades, uncertainty quantification (UQ) of various parametric PDE problems has been actively studied. An incomplete list of such works would be \cite{para1,para2,para3,para4,para5, para6, para7, para8, para9, sch1, alexey, alex2019, alex20201,andreev2010, kien2022, semi2013} and the references therein. In this paper, we will focus on elliptic PDEs, especially elliptic semilinear eigenvalue problems. Parametric PDEs have broad applications, for example, to physics, biology, engineering, etc. Its parameters can be seen in two different ways, deterministic and stochastic.  From a deterministic viewpoint, we can think our problem as tuning the parameters $(y_i)_{i\geq 1}$ which is equivalent to tuning the influences of factors $(b_i)_{i\geq 1}$. In this way, it is possible to study the behavior of the solution as we tune the coefficient functions $(b_i)_{i\geq 1}$. For example, we can think of studying a super conductor in which possibly many different kinds of tunable {potentials} are available to influence on the conductor. In this case, we can tune the influences of each potential in order to study the behaviors of the critical temperature in different combinations of potentials. From a stochastic viewpoint, when many factors are randomly affecting our solution, the potential function $b(\bm{y})$ in our problem can be considered as a random variable. The motivation of choosing the affine dependence form of our parameters is from Karhunen-Lo\'eve (KL) expansion and the detailed explanation of this motivation can be found in the section 2 of \cite{schbest}. We note that, by Kosambi-Karhunen-Lo\'eve theorem, any square integrable centered stochastic process has a KL expansion. For the detail of KL expansion, see \cite{kar}.

 Among various parametric PDEs, UQ of parametric elliptic PDE problems has been studied extensively in various ways. Mentioning a few such studies, we can find the studies about the Quasi-Monte Carlo method in \cite{alex2019,alex20201, sch1,kien2022,para5, para6, para8}, best N-term approximations in  \cite{schbest}, Deep neural network in \cite{zech,mishra}, and the references therein. A good introduction to affine dependence parametric PDEs can be found in \cite{Cohen2015}. Also, semilinear elliptic PDEs have been studied in \cite{sch4}.

 In studying parametric PDEs, we are primarily interested in simultaneously solving a class of PDEs.  Thus, the study of its Taylor generalized polynomial chaos (gpc) is crucial, and the precise estimation of the mixed derivatives or, equivalently, Taylor coefficients provides us the information on the convergence. For example, one of the recent results shows that the convergence of the sparse gpc Taylor expansion of the solution in $\ell^r$ with $r\in (0,1)$ provides us deep neural network approximation (Theorem 4.9 in \cite{zech}).

 Considering its ubiquity and importance in physics and engineering, the parametric elliptic eigenvalue problem has been only recently studied, and there is still room for improvement. Since the study of \cite{andreev2010} in 2010, there have been a few studies, for example, in  \cite{alex2019, alex20201, alexey, kien2022}. In \cite{alex2019, alex20201} with all the coefficients parametrized, the form of  $C(|\bm{\nu}|!)^{1+\varepsilon}\bm{\beta}^{\bm{\nu}}$ is obtained as the upper bound for $|\partial^{\bm{\nu}}\lambda(\bm{y})|$ and $\|\partial^{\bm{\nu}}u(\bm{y})\|_{H_0^1}$ for some constants  $C>0$ and an arbitrary $\varepsilon\in (0,1)$. Although it is not enough for convergence of the Taylor gpc expansion, \cite{alex2019} showed that this result is enough for showing dimension-independent error bound for estimating the expectation of the solution under the assumption that the sequence $\bm{\beta}\in \ell^r$ with $r\in (0,1)$. The only obstacle generated by the existence of $\varepsilon$ was the analysis of the case of $r=1$.   Recently, \cite{alexey} obtained an improved estimation. Specifically, they used a falling factorial technique to remove $\varepsilon>0$ from the estimation by \cite{alex2019} and extended the class of parametrizations.

We note that, for the parametric linear and nonlinear elliptic PDEs, \cite{breaking,semi2013} showed that the two PDEs have the same form of upper bound estimations for the mixed derivatives. Following the works mentioned above, we study nonlinear eigenvalue problems with power type nonlinearities comparing with the result of linear eigenvalue problems in \cite{alexey}. Also, this problem has its own interest in physics. For example, a famous example of our class of equations is the Gross-Pitaevskii equation (when $p=3$ and $a=1$) which is known to describe superfluidity and superconductivity. Our setting can be seen as describing such phenomena with random potential. Further explanation of this equation can be found in \cite{gpe,gpe2} and an incomplete list of the related studies is \cite{lieb2000, weizhu1, weizhu2}.  
\subsection{The summary of our contribution.} \label{summary}
 In this paper, to the best of our knowledge, we make the first attempt at UQ analysis of the parametric nonlinear eigenvalue problem. The analyticity of the eigenpairs with respect to the parameters is necessary for estimating the mixed derivatives. The studies of the linear eigenvalue problems so far, in \cite{alex2019, alexey, kien2022}, strongly depend on the analyticity result by \cite{andreev2010} which is deeply rooted in the linear operator perturbation theory by Kato in \cite{kato}. However, the linear theory by Kato cannot be directly applied to our nonlinear problems. Thus, in this paper, we use the technique using the implicit function theorem introduced in \cite{nonlinear}. After the careful analysis of \eqref{gpe}, we apply the implicit function theorem technique to show the analyticity of the eigenpairs. This analyticity allows us to take arbitrarily high order derivatives of the eigenpair and allows us to apply the method of mathematical induction used in \cite{alex2019, alexey} to estimate the upper bound for the mixed derivatives. Although we use a similar argument, it involves different techniques that use the properties of the nonlinear eigenvalue, also shown in this paper. One of the key ingredients of our analysis is to show that the smallest eigenvalue is uniformly bounded away from that of a different linear elliptic operator. And, by using recently found technique in \cite{alexey}, we show that the norms of each mixed derivative, $|\partial^{\bm{\nu}}\lambda(\bm{y})|$, $\left\|\nabla \partial^{\bm{\nu}}u(\bm{y})   \right\|_{L^2}$ and $|\partial^{\bm{\nu}}\mathcal{E}(\bm{y})|$ are bounded above by the same form with \cite{alexey}.
 
 With the bounds, considering the parameters as uniformly distributed random variables, we obtain an identical QMC error convergence rate for approximating $\mathbb{E}_{\bm{y}}[\lambda(\bm{y})]$ and $\mathbb{E}_{\bm{y}}[\mathcal{G}(u)(\bm{y})]$ for a given $\mathcal{G}\in H^{-1}$ with that of the parametric linear eigenvalue problem studied in \cite{alex2019}. Also, as we removed $\varepsilon$ from the estimation of the bound, we also have an additional case for $(\|b_i\|_{L^\infty})_{i\in \mathbb{N}}\in \ell^1$. In this process, we point out that any analytic function with such an upper bound on the mixed derivatives has the same QMC error convergence rate.

\subsection{Organization.}
\label{organization}  In section \ref{preliminary}, we investigate the properties of the nonparametric eigenpair that will be used in our parametric analysis. In section \ref{main section}, using the result from section \ref{preliminary}, we develop some parametric tools such as analyticity of the eigenpair and uniform differences of eigenvalues. We use them to estimate the bound for the mixed derivatives of the ground eigenpairs and the energy. In section \ref{application}, we discuss UQ analysis for approximating the expectation of the functions of our interest using the Quasi-Monte Carlo method.

\subsection{Notation.}
\label{notation} Throughout this paper, we are mostly interested in the dependence on the stochastic parameter $\bm{y}$ rather than the spatial variable $x$. Thus, we drop the dependence of $x$ for convenience and readability. In other words, we use $u(\bm{y})$ instead of $u(x,\bm{y})$. When the dependence on $\bm{y}$ is clear from the context without any confusion, we sometimes drop $\bm{y}$. Also, unless otherwise stated, we drop $\Omega$. For example, we use $L^2$, $\int u$ and $H_0^1$ instead of $L^2(\Omega)$, $\int_\Omega u(x) dx$ and $H_0^1(\Omega)$. Here, $L^r$ for $r\geq 1$ and $H_0^1$ is the space of Lebesgue measurable functions with finite $\|\cdot\|_{L^r}$ and $\|\cdot\|_{H_0^1}$ norm respectively. Those norms are defined by:
\begin{align}
\|u\|_{L^r}:=
\begin{cases}
\left( \int_{\Omega} |u|^r \right)^{\frac{1}{r}}, &r<\infty,\\
\underset{x\in \Omega}{\mathrm{esssup}} |u(x)|, &r=\infty.
\end{cases}
\end{align} 
\begin{align}
\|u\|_{H_0^1}:=\|\nabla u\|_{L^2}.
\end{align}
The inner product for $L^2$ is notated by $\left<\cdot,\cdot \right>$ which is clearly defined by $\left<f,g 
\right>=\int fg$. The set $C^k$ in this paper is the space of $k$-times continuously differentiable functions.
  We use boldface letters for elements in $\mathbb{R}^\mathbb{N}$. Especially, we use $\bm{y}=(y_1,y_2,\dots)\in U$ for stochastic parameter and $\bm{\nu}=(\nu_1,\nu_2,\dots)$ for multi-index. We notate $|\bm{\nu}|=\sum_{i\geq 1} \nu$.  Let $\mathcal{F}=\{\bm{\nu}\in \mathbb{N}^\mathbb{N}: |\bm{\nu}|<\infty \}$. Here we used $\mathbb{N}$ as the set of natural numbers including $0$.  For multi-index power or derivative, for any $\bm{\nu}\in \mathcal{F}$, for any analytic function $f:U\rightarrow \mathbb{C}$ and $\bm{y}\in U$, we denote 
   \begin{align}
   \partial^{\bm{\nu}} f(\bm{y})=\frac{\partial^{|\bm{\nu}|}f\ }{\partial y_1^{\nu_1}\partial y_2^{\nu_2}\cdots }(\bm{y}),
   \end{align}
   and
\begin{align}
\bm{y}^{\bm{\nu}}=\prod_{i\geq 1}y_i^{\nu_i}.
\end{align}   
 As for multi-index combination notation, we use 
 \begin{align}
 \begin{pmatrix}
 \bm{\nu}\\\bm{m}
 \end{pmatrix}=\prod_{j\geq 1} 
 \begin{pmatrix}
 \nu_j\\m_j
 \end{pmatrix}.
 \end{align}
In the case when there are more than three multi-indexes involved, we use
\begin{align}
\begin{pmatrix}
&\bm{\nu}\\
\bm{m}_{1} &\cdots &\bm{m}_p
\end{pmatrix}
=
\prod_{j\geq 1}
\begin{pmatrix}
&\nu_j\\
m_{1j} &\cdots &m_{pj}
\end{pmatrix}:=\prod_{j\geq 1}
\frac{\nu_j!}{m_{1j}!m_{2j}!\cdots m_{pj}!}.
\end{align} 
Throughout this paper, we will use the notation 
\begin{align}
\sum_{\substack{\sum_{j=1}^{p} \bm{m}_j=\bm{\nu} \\ 
0\neq \bm{m}_j<\bm{\nu}}}:=\sum_{\substack{\sum_{j=1}^{p} \bm{m}_j=\bm{\nu} \\ 
0\neq \bm{m}_j<\bm{\nu},\text{for all } j}},
\end{align}
for the summations over many parameters. In other words, we will drop `for all $j$' for our convenience.   
As for the basic operations of multi-indexes, we define $\bm{\nu}+\bm{m}=(\nu_i+m_i)_{i\geq 1}$ and $a\bm{\nu}=(a\nu_i)_{i\geq 1}$ for any constant $a\in \mathbb{R}$. Also,  we say $\bm{\nu}\leq \bm{m}$ if and only if $\nu_i\leq m_i$ for all $i\geq 1$. The relation $<$ is defined accordingly. Lastly, for a given positive integer $i$, we use $\bm{e}_i$ as a multi-index whose entries are all zero except the $i$-th entry being $1$.
\section{Preliminaries}\label{preliminary}

In this section, we investigate some properties of the ground state of the nonparametric problem that will be used in the next section. In order for this, we consider the following problem:
\begin{align}
\label{semi1}
\begin{cases}
-\nabla\cdot(a\nabla)u+ bu+\eta |u|^{p-1}u=\lambda u, &\text{ on }\Omega,\\
u=0, &\text{ on }\partial \Omega,
\end{cases}
\end{align}
where all the notations and assumptions are consistent with  \eqref{gpe}. To be specific, our analysis in this section does not depend on the parameter, $\bm{y}$, but we still stick to the assumption that $a\in C^1$ and $ b\in L^\infty(\Omega)$. Also, we assume that there exists a constant $a_{min}>0$ such that $a(x)\geq a_{min}$, and that $b(x)\geq 0$ for all $x\in \Omega$. Also, we further assume that $\Omega$ has $C^2$ boundary. Throughout this section, whenever $(u,\lambda)$ is defined to be the solution to \eqref{semi1}, we assume all the relevant assumptions as well. 

Throughout this paper, the Sobolev embedding $H_0^1\hookrightarrow L^q$ with $q$ determined by the power $p$ in the nonlinearity plays an important part. For example, it is necessary to allow the nonlinearity $|u|^{p-1}u$ to be in $L^2$. Thus, we define two collections of admissible pairs of the dimensions and the powers for our convenience as follows:
\begin{align}
\label{dp}
\mathcal{A}&:=\left\{ (d,p)\in \mathbb{N}\times \mathbb{N}: (d,p)\in [1,2]\times [1,\infty)\text{ or } p\in \left[1, \frac{d}{d-2}  \right] \right\},
\\
\label{dp2}
\mathcal{A}'&:=\left\{ (d,p)\in \mathbb{N}\times \mathbb{N}: (d,p)\in [1,2]\times [1,\infty)\text{ or } p\in \left[1, \frac{2}{d-2}  \right] \right\}.
\end{align}
Note that $\mathcal{A}'\subset \mathcal{A}$. We will use $\mathcal{A}$ for the ground eigenpairs in Theorem \ref{main} and will use $\mathcal{A}'$ for the ground energy in Corollary \ref{esen}. The constants $\frac{d}{d-2}$ and $\frac{2}{d-2}$ above comes from $2p\leq \frac{2d}{d-2}$, $2(p+1)\leq \frac{2d}{d-2}$ and the fact that $L^q\hookrightarrow H_0^1$ for $q\in \left[2,\frac{2d}{d-2} \right]$ when $d\geq 3$.

 Note that when $p=1$, the problem is linear, so the existence and uniqueness of the solution are well known by the Lax-Milgram theorem. Thus, we assume without loss of generality that $p\geq 2$. Before stating the main result of this section, let us define the ground state of the problem \eqref{semi1}. We first let $H=\left\{ v\in H_0^1 : \|v\|_{L^2}=1 \right\}$ and then define a functional $E:H_0^1\rightarrow \mathbb{R}$ by
 \begin{align}
 \label{energy}
 E[u]:=\int a|\nabla u|^2 +\int b|u|^2 +\frac{2\eta}{p+1}\int |u|^{p+1}.
 \end{align}
Then we can define the ground energy $\mathcal{E}$ by
\begin{align}
\label{genergy}
\mathcal{E}:=\inf_{u\in H} E[u].
\end{align}
The minimizer $u$ of \eqref{genergy} is the ground state and it solves \eqref{semi1} with corresponding eigenvalue $\lambda=\mathcal{E}+\eta \frac{p-1}{p+1}\int |u|^{p+1}$ (Lemma \ref{sol}). This can be seen by the calculus of variation techniques (for example, see chapter 11 in \cite{lieb1997} or \cite{lieb2000}). We call the pair $(u,\lambda)$, the ground eigenpair, and we call $\mathcal{E}$ the ground (or minimal) energy. 

 Note that due to the modulus in \eqref{energy}, the minimizer is unique up to the modulus $|\cdot|$ (Lemma \ref{M}). In other words, if $u$ is a solution, then $e^{i\theta}u$ is also a solution for any $\theta\in \mathbb{R}$ which implies the non-uniqueness of the solution. This obstacle can be avoided by noticing that $|u|$ is a unique non-negative real solution (Corollary \ref{absol}) and focusing only on this solution for the future arguments. It allows us to work on the analytic formation of our PDE, which is \eqref{semi}. All the mentioned results and all the tools we need for the following sections are collected in the following theorem, whose proof is provided in Appendix \ref{appendixa}.
 \\
 
\begin{theorem}
\label{exiuni}
There exists a ground eigenpair $(u,\lambda)$ of \eqref{semi1} where $u$ is unique up to the modulus  $|\cdot|$. Moreover, if $(d,p)\in \mathcal{A}$, the unique non-negative real-valued solution $|u|$ is indeed strictly positive in $\Omega$. 
 \end{theorem} 
 \bigskip
 
 Significantly, the result of the strict positivity of the solution stated in the Theorem \ref{exiuni} (or Lemma \ref{sp}) will be a useful tool in many ways especially when proving the uniform positive gap between two operators $\mathcal{O}$ and $\mathcal{T}$ defined in the next section. By Theorem \ref{exiuni}, we can also say that \eqref{semi} has a unique real positive solution that is also a solution to \eqref{gpe}. From now on, we will only consider the strictly positive real solution to \eqref{gpe}, and it will allow us to remove the modulus sign in the nonlinear term of \eqref{gpe} and it transforms it to the analytic problem \eqref{semi}.

\section{Parametric semilinear eigenvalue problem}\label{main section}

By the result of the previous section, instead of \eqref{gpe}, we can focus on the ground state of the following analytic form of parametric semilinear elliptic eigenvalue problem :
 \begin{align}
 \label{semi}
 \begin{cases}
 -\nabla\cdot(a(x)\nabla )u(x,\bm{y})+{b(x,\bm{y})}u({x},\bm{y})+\eta u({x},\bm{y})^p=\lambda u({x},\bm{y}), &({x},\bm{y})\in \Omega\times U,\\
u({x},\bm{y}) =0, &({x},\bm{y})\in \partial \Omega \times U,
 \end{cases}
 \end{align}
where we assume the following: 

\begin{assumption}
\label{assumption} 

\begin{itemize}
\item $\Omega\subset \mathbb{R}^d$ is bounded open domain with $C^2$ boundary,

\item The diffusion coefficient $a$ is a $C^1$ function and there exists constants $a_{min}>0$ such that $a(x)\geq a_{min}$ for all $x\in \Omega$,

\item $b(\bm{y})$ has the following form:
 \begin{align}
b(\bm{y})=b_0+\sum_{j\geq 1}y_j b_j,
 \end{align}

where $ (b_j)_{j\geq 0}\subset L^\infty$ with $(\|b_j\|_{L^\infty})_{j\geq 0}\in \ell^1(\mathbb{N})$ such that and $b(\bm{y})\geq 0$ for all $(x,\bm{y})\in \Omega\times U$,

\item The constant $\eta$ is positive real number, and $U=\left[ -\frac{1}{2},\frac{1}{2} \right]^{\mathbb{N}}$ is the parameter space,

\item The pair $(d,p)$ belongs to $\mathcal{A}$.
\end{itemize}

\end{assumption}
\bigskip 

We can easily see that the third assumption implies  $b(\bm{y})\in L^\infty $ for all $\bm{y}\in U$. From now, whenever we say $(u(\bm{y}),\lambda(\bm{y}))$ is the ground state of problem \eqref{semi}, Assumption \ref{assumption} will be automatically assumed as well. The main goal of this section is to estimate the mixed derivatives of the ground eigenpair. For this goal, we start this section by showing that the ground eigenpair admits the mixed derivatives.

\subsection{Parametric analyticity of the ground state.}In this subsection, we will verify the analyticity of the ground state, $(\lambda(\bm{y}),u(\bm{y}))$, with respect to the parameters and will investigate their properties. In order to estimate the mixed derivatives of the ground eigenpair, the analyticity should be justified first. Without loss of generality, we assume $p\geq 2$.  From now on, we denote $\lambda(\bm{y})$ to be the smallest eigenvalue and $u(\bm{y})$ to be the corresponding strictly positive eigenfunction to \eqref{semi} under the Assumption \ref{assumption}. For convenience, with the ground state $u(\bm{y})$ and under the Assumption \ref{assumption},  we use the following notations for different but related linear operators:
\begin{align}
\label{o}
\mathcal{O}(\bm{y})&:=-\nabla \cdot (a\nabla)+b(\bm{y})+\eta u(\bm{y})^{p-1},
\\
\label{t}
{\mathcal{T}}(\bm{y})&:=-\nabla \cdot (a\nabla)+b(\bm{y})+p\eta u(\bm{y})^{p-1}.
\end{align}
We use notation $u_T(\bm{y})$ and $\lambda_T(\bm{y})$ to denote the ground eigenpair of $\mathcal{T}(\bm{y})$. We first show the following useful lemma:
\\

\begin{lemma}
\label{strictpositive}
For any $\bm{y}\in U$, if $(u(\bm{y}),\lambda(\bm{y}))$ is the ground eigenpair of \eqref{semi}, then it is the ground eigenpair of the operator $\mathcal{O}(\bm{y})$. 
\end{lemma}
\begin{proof}
Let $\bm{y}\in U$ be given. It is clear that $(u(\bm{y}),\lambda(\bm{y}))$ is an eigenpair of the linear operator $\mathcal{O}(\bm{y})$. The pair $(u(\bm{y}),\lambda(\bm{y}))$ is indeed the ground eigenpair of $\mathcal{O}(\bm{y})$. If not, there is a pair $(v(\bm{y}),\mu(\bm{y}))\neq (u(\bm{y}),\lambda(\bm{y}))$ that is the ground eigenpair of $\mathcal{O}(\bm{y})$. Then $v(\bm{y})$ is strictly positive over $\Omega$, for example, by Theorem 1.2.5 in \cite{henrot} with the boundedness of $u$ proved in \ref{sp}. Now, note that, by the orthogonality of eigenfunctions for $\mathcal{O}(\bm{y})$, we have $\left< u(\bm{y}),v(\bm{y}) \right>=0$. However, by the strict positivity of $v(\bm{y})$ and $u(\bm{y})$, the equality, $\left<u(\bm{y}),v(\bm{y}) \right>=\int u(\bm{y})v(\bm{y})=0$ is impossible which shows a contradiction. Therefore, $(u(\bm{y}),\lambda(\bm{y}))$ is the ground eigenpair of the elliptic operator $\mathcal{O}(\bm{y})$.
 
\end{proof}
\bigskip

Using a part of the argument above, we can see that $u_T(\bm{y})$ is also strictly positive. With this fact, we can show the analyticity of the ground state as follows.\\

\begin{definition} 
Let $B$ be a given Banach space and let $O\subset \mathbb{C}^{\mathbb{N}}$ be a domain.
 A function $f: O\rightarrow  B$  is called\textit{ separately complex analytic} if $f$ is analytic for each coordinate.
\end{definition}
\bigskip

\begin{theorem}
\label{anal}
Suppose that  $(u(\bm{y}),\lambda(\bm{y}))$ is the ground eigenpair of \eqref{semi} with parametrized coefficient functions $a(\bm{y})$ and $b(\bm{y})$ which are analytic in $\bm{y}$. Then the  ground eigenpair  $(u(\bm{y}),\lambda(\bm{y}))$ is separately complex analytic at each $\bm{y}\in U$, i.e, for each coordinate $y_i$, there exists an open set $O_i\subset \mathbb{C}$ containing $y_i$ such that $(u(y_i),\lambda(y_i))$ is complex analytic on $O_i$ with other coordinates fixed. Furthermore,  for any $M\subset \mathbb{N}$ with $|M|<\infty$, the ground state is jointly complex analytic on $\prod_{i\in M}O_i$.
 \end{theorem}
\bigskip 

\begin{proof}
 Let $\bm{y}\in U$ and $i\geq 1$ be given. Let us define an operator $\mathcal{N}(\xi,\ell,f)={(-\nabla\cdot( a}(\bm{y}+\xi \bm{e}_i) \nabla)+b(\bm{y}+\xi \bm{e}_i)+\eta(u(\bm{y})+f)^{p-1}-\ell)(u(\bm{y})+f)$ defined on $\mathbb{C}\times \mathbb{C}\times E$ where $E=span\{\psi\}^{\perp}$ in $H_0^1$ and $\psi$ will be specified later. Then note that $\mathcal{N}$ is complex analytic in $h$, $\lambda$ and $f$ due to the polynomial dependence. Also, observe that, for each $(\mu,w)\in \mathbb{C}\times E$,
 \begin{align}
 &D_{(\ell,f)}\mathcal{N}(0,\lambda(\bm{y}),0)
 (\mu,w)\notag
 \\
 &=
 \lim_{h\rightarrow 0} \frac{\mathcal{N}(0,\lambda(\bm{y})+h\mu,hw)-\mathcal{N}(0,\lambda(\bm{y}),0)}{h}\notag
 \\
 &=
 \lim_{h\rightarrow 0} \frac{ (-\nabla\cdot (a(\bm{y}) \nabla)+b(\bm{y})+\eta(u(\bm{y})+hw)^{p-1}-\lambda(\bm{y})-h\mu)(u(\bm{y})+hw)}{h}   \notag
\\
&=\lim_{h\rightarrow 0}
\frac{(-\nabla\cdot(a(\bm{y})\nabla)+b(\bm{y})-\lambda(\bm{y})-h\mu)(u(\bm{y})+hw)+\eta(u(\bm{y})+hw)^{p}-\eta u(\bm{y})^p+\eta u(\bm{y})^p}{h}
 \notag\\
&=\lim_{h\rightarrow 0}
 \frac{(-\nabla\cdot(a(\bm{y})\nabla)+b(\bm{y})-\lambda(\bm{y})-h\mu)hw-h\mu u(\bm{y})+\eta(u(\bm{y})+hw)^{p}-\eta u(\bm{y})^p}{h}
 \notag 
\\
&=
{ (-\nabla\cdot (a(\bm{y}) \nabla)+b(\bm{y})-\lambda(\bm{y}))w-\mu u(\bm{y})+\eta\lim_{h\rightarrow 0}
 \frac{ (u(\bm{y})+hw)^{p}- u(\bm{y})^p}{h} }
\notag 
\\
 &=(-\nabla\cdot (a(\bm{y}) \nabla)+b(\bm{y})-\lambda(\bm{y}))w-\mu u(\bm{y})
 +p\eta  u(\bm{y})^{p-1}w
 \notag 
\\
 &=
 (-\nabla\cdot (a(\bm{y}) \nabla)+b(\bm{y})+p\eta u(\bm{y})^{p-1}-\lambda(\bm{y}))w-\mu u(\bm{y})\notag
 \\
 &=
 \left( \mathcal{T}(\bm{y})-\lambda(\bm{y})\right) w - \mu u(\bm{y}),
 \end{align}
 where we used the fact that $\mathcal{N}(0,\lambda(\bm{y}),0)=0$ in the second equality, {and used the equation $\mathcal{O}u(\bm{y})-\lambda(\bm{y})=0$ in the fourth equality}.
We claim that $\lambda(\bm{y})$ is not an eigenvalue of $\mathcal{T}(\bm{y})$. If $\lambda(\bm{y})$ is an eigenvalue of the linear operator $\mathcal{T}(\bm{y})$, then there should exist the corresponding eigenfunction $v(\bm{y})$ such that 
\begin{align}
\label{lTlessl}
\lambda_{T}(\bm{y}) \leq \left< \mathcal{T}(\bm{y})v(\bm{y}),v(\bm{y})\right>=\lambda(\bm{y}).
\end{align}
However, by the strict positivity of $u(\bm{y})$ and $u_T(\bm{y})$,
\begin{align}
\label{lslesslT}
\lambda(\bm{y})&\leq \left< \mathcal{O}u_T(\bm{y}),u_T(\bm{y}) \right>\notag \\
& < \left< \mathcal{O}u_T(\bm{y}),u_T(\bm{y}) \right> +(p-1)\eta\int u(\bm{y})^{p-1}u_T(\bm{y})^2\notag \\
& =\left<\mathcal{T}(\bm{y})u_T(\bm{y}),u_T(\bm{y}) \right>\notag\\
&= \lambda_T(\bm{y}),
\end{align}
 where we used the fact that $u(\bm{y})$ is the ground state of the linear operator $\mathcal{O}(\bm{y})$ with the ground eigenvalue $\lambda(\bm{y})$ proved in Lemma \ref{strictpositive} in the first inequality. Note that the two inequalities \eqref{lTlessl} and \eqref{lslesslT} make a contradiction. Thus, $\lambda(\bm{y})$ cannot be an eigenvalue of $\mathcal{T}(\bm{y})$. {Then it is clear that}  $\mathcal{T}(\bm{y})-\lambda(\bm{y}): H_0^1\rightarrow H^{-1}$ is an isomorphism.
 
  Now, let $\psi\in H_0^1$ be such that $(\mathcal{T}(\bm{y})-\lambda(\bm{y}))\psi=u(\bm{y})$. Since we can consider $(\mathcal{T}(\bm{y})-\lambda(\bm{y}))$ to be linearly acting on $\psi$, {we can restrict the operator so that} it is now an isomorphism between $E$ and $span\{u(\bm{y})\}^{\perp}\subset H^{-1}$. It implies that $D_{\lambda,f}\mathcal{N}(0,\lambda(\bm{y}),0): {\mathbb{C}}\times E\rightarrow {span\{u(\bm{y})\}^\perp\subset H^{-1}}$ is an isomorphism. {The boundedness of the inverse operator is clear. Indeed,  observe that, with the varying generic constant $C$,
\begin{align}
\|(\mu,w)\|_{\mathbb{R}\times H_0^1}^2
&= |\mu|^2+\|w\|_{H_0^1}^2
\notag\\
&\leq \frac{1}{\|u(\bm{y})\|^2_{H^{-1}}} \|\mu u(\bm{y})\|^2_{H^{-1}}+C\|(\mathcal{T}(\bm{y})-\lambda(\bm{y}))w\|_{H^{-1}}^2
\notag\\
&\leq C\left( \|\mu u(\bm{y})\|^2_{H^{-1}}+\|(\mathcal{T}(\bm{y})-\lambda(\bm{y}))w\|_{H^{-1}}^2  \right)
\notag\\
&= C\|(\mathcal{T}(\bm{y})-\lambda(\bm{y}))w -\mu u(\bm{y})  \|_{H^{-1}}^2,
\end{align}  
where, for the first inequality, we used the fact that $\|u(\bm{y})\|_{H^{-1}}\neq 0$ because $u(\bm{y})\neq 0$ followed from $\|u(\bm{y})\|_{L^2}=1$, and the fact that $(\mathcal{T}(\bm{y})-\lambda(\bm{y}))$ is an isomorphism so there is such a constant $C>0$. The last equality follows from the orthogonality between the {image of $E$ under  $(\mathcal{T}(\bm{y})-\lambda(\bm{y}))$} and $span\{u(\bm{y})\}$ in $H^{-1}$.} {To be specific, recall that the mapping $(\mathcal{T}(\bm{y})-\lambda(\bm{y})):\mathbb{C}\times E \rightarrow span\{u(\bm{y})\}^\perp\subset H^{-1}$ is an isomorphism, and $span\{u(\bm{y})\}^\perp$ is orthogonal complement of span $span\{u(\bm{y})\}$ in the sense of $H^{-1}$ norm.}  Therefore, by implicit function theorem, there exists an open neighborhood, $O_i\subset \mathbb{C}$, of $0$  such that we have complex analytic functions $\ell: O_i\rightarrow \mathbb{R}$ and $f: O_i\rightarrow E$ such that $\mathcal{N}(\xi,\ell(\xi),f(\xi))=0$ for all $\xi\in O_i$ and $\ell(0)=\lambda(\bm{y})$, $f(0)=0$. Furthermore, the ground states are unique for each ${\xi}\in O_i$. Therefore, $u(\bm{y}+\xi\bm{e}_i)=u(\bm{y})+f(\xi)$ and so $u:O_i\rightarrow H_0^1$ is also complex analytic in $i$ th coordinate. The last statement is clear by Hartog's Theorem.
\end{proof}
\bigskip

 Note that Theorem \ref{anal} is stronger than what we need for our problem \eqref{semi}. For example, in the theorem, the diffusion coefficient $a$ is parametrized, and $a(\bm{y})$ and $b(\bm{y})$ are only assumed to be analytic in $\bm{y}$. {In addition, in the  case of $\eta =0$, the operator $(-\nabla\cdot(a(\bm{y})\nabla)+b(\bm{y})-\lambda(\bm{y}))$ is still an isomorphism between $span\{\psi\}^{\perp}\subset H_0^1$ with $\psi=u(\bm{y})$ and $span\{u(\bm{y})\}^{\perp}\subset H^{-1}$ so this argument} also works for linear eigenvalue problem. With this regularity of the ground eigenpair, now, we are allowed to take the derivatives of order $\bm{\nu}$ for any $\bm{\nu}\in \mathcal{F}$. As a first observation of this result, we find the representation of the mixed derivative of the nonlinear term that will be useful in the following results.
\\

\begin{lemma}
\label{gp}
For any multi-index $\bm{\nu}\in \mathcal{F}$, { any analytic function $u$ that admits $\partial^{\bm{\nu}}u$ for any $\bm{\nu}\in \mathcal{F}$, and any positive integer $p$, the following equality holds:}
\begin{align}
\label{gprep}
\partial^{\bm{\nu}}(u^p)=pu^{p-1}\partial^{\bm{\nu}} u  +\sum_{n=0}^{p-2}
{\begin{pmatrix}
p\\
n
\end{pmatrix}}
u^n
\sum_{\substack{\sum_{j=1}^{p-n} \bm{m}_j=\bm{\nu} \\ 
0\neq \bm{m}_j<\bm{\nu}}}
\begin{pmatrix}
&\bm{\nu}\\
\bm{m}_1&\cdots &\bm{m}_{p-n}
\end{pmatrix}
\prod_{i=1}^{p-n} \partial^{\bm{m}_i}u.
\end{align}
\end{lemma}

\begin{proof}
{
First of all, for any positive integer $p$ and $n\in [p]:= \{1,2,\dots, p\}$, let 
\begin{align*}
A_p&=\left\{ (\bm{m}_i)_{i=1}^{p} : \sum_{i=1}^{p}\bm{m}_i = \bm{\nu}, \text{ }\bm{m}_i< \bm{\nu}\text{ for all } i \right\},\\
B_p^n&=\left\{(\bm{m}_{i})_{i=1}^{p} \in A_p : \bm{m}_n=0  \right\}.
\end{align*}}

\textcolor{black}{
Now, observe that we can classify the sequence in $A_p$ as follows:
{
\begin{align}
\label{classification}
A_p
&=
B_p^1\dot{\cup} (B_p^1)^c
\notag\\
&=
(B_p^1\cap B_p^2)
\dot{\cup} (B_p^1\cap (B_p^2)^c)
\dot{\cup} 
((B_p^1)^c\cap B_p^2)
\dot{\cup} ((B_p^1)^c\cap (B_p^2)^c)
\notag\\
&=
(B_p^1\cap B_p^2\cap B_p^3)
\dot{\cup} (B_p^1\cap B_p^2\cap (B_p^3)^c)
\dot{\cup} (B_p^1\cap (B_p^2)^c\cap B_p^3)
\dot{\cup} (B_p^1\cap (B_p^2)^c\cap (B_p^3)^c)
\notag\\
&
\dot{\cup} ((B_p^1)^c\cap B_p^2\cap B_p^3)
\dot{\cup} ((B_p^1)^c\cap B_p^2\cap (B_p^3)^c)
\dot{\cup} ((B_p^1)^c\cap (B_p^2)^c\cap B_p^3)
\notag\\
&
\dot{\cup} ((B_p^1)^c\cap (B_p^2)^c\cap (B_p^3)^c)
\notag\\
&=\cdots
\notag\\
&=
\dot{\bigcup_{0\leq n\leq p-2}} 
\left( 
{{\bigcup_{\pi\in \Pi_{[p]}}}} 
\left( 
\bigcap_{j=1}^n B_p^{\pi(j)}
\cap 
\bigcap_{j=n+1}^p (B_p^{\pi(j)})^c
\right)
\right),
\end{align}
}
where $\dot{\cup}$ denotes a disjoint union, $(B_p^n)^c=A_p\setminus B_p^n$, and  $\Pi_{[p]}$ is the set of all possible permutations of $[p]:=\{1,\dots, p\}$. To be specific, for each $n\in \left\{ 1,2,\dots, p-2\right\}$, the inner union is collection of all sequences of multi-indices with $n$ number of zero and $p-n$ number of nonzero indices. The case of $n=p$ is removed because it is when all sequences are zero which is not possible because of the summation constraint $\sum_{i=1}^{p}\bm{m}_i=\bm{\nu}$. The case of $n=p-1$ is removed because of the condition that $\bm{m}_i<\bm{\nu}$ for all $i$.   Now, for any $n\in \{0,1,\dots, p-2\}$ and for any  $\pi\in\Pi_{[p]}$, we use following notations for our convenience:
{
\begin{align}
\label{B_pi}
B(n,\pi):=\bigcap_{j=1}^n B_p^{\pi(j)}
\cap 
\bigcap_{j=n+1}^p (B_p^{\pi(j)})^c.
\end{align}}
}
{
{
 Note that, for given $\pi\in \Pi_{[p]}$ and $k\in [p]$, we have $\bm{m}_k=0$ for all $(\bm{m}_i)_{i=1}^p\in {B(n,\pi)}$ or $\bm{m}_k\neq 0$ for all $(\bm{m}_i)_{i=1}^p\in {B(n,\pi)}$. Also, note that, for any $\sigma,\pi\in \Pi_{[p]}$,  observe that}
\begin{align}
\label{symmetric}
\sum_{(\bm{m}_i)_{i=1}^{p}\in {B(n,\pi)}} 
\begin{pmatrix}
&\bm{\nu}\\
\bm{m}_1&\cdots &\bm{m}_p
\end{pmatrix}
\prod_{j=1}^{p} \partial^{\bm{m}_j}u
&=
u^n\sum_{(\bm{m}_i)_{i=1}^{p}\in A_{p-n}' } 
\begin{pmatrix}
&\bm{\nu}\\
\bm{m}_1&\cdots &\bm{m}_{p-n}
\end{pmatrix}
\prod_{j=1}^{p-n} \partial^{\bm{m}_j}u
\notag\\
&=
\sum_{(\bm{m}_i)_{i=1}^{p}\in {B(n,\sigma)}} 
\begin{pmatrix}
&\bm{\nu}\\
\bm{m}_1&\cdots &\bm{m}_p
\end{pmatrix}
\prod_{j=1}^{p} \partial^{\bm{m}_j}u,
\end{align}
where, for any positive integer $n$,
\begin{align}
A_{n}'=\left\{ (\bm{m}_{i})_{i=1}^{n}\in A_n: \bm{m}_i\neq 0 \text{ for all }i \right\}.
\end{align}
Now, it is easy to see that, for given $n\in \{0,\dots, p-2\}$, in the inner union of the last expression of \eqref{classification}, there are $\frac{p!}{n!(p-n)!}$ disjoint components. Applying the properties above, {for arbitrary permutation $\pi\in \prod_{[p]}$,} we have the following expression:
\begin{align}
\label{last}
\partial^{\bm{\nu}}(u^p)
&=
{\sum_{\sum_{i=1}^{p}\bm{m}_i=\bm{
\nu}}
\begin{pmatrix}
&\bm{\nu}\\
\bm{m}_1&\cdots &\bm{m}_p
\end{pmatrix}
\prod_{j=1}^{p} \partial^{\bm{m}_j}u}
\notag\\
&=
pu^{p-1}\partial^{\bm{\nu}}u
+
\sum_{(\bm{m}_i)_{i=1}^{p} \in A_p}\begin{pmatrix}
&\bm{\nu}\\
\bm{m}_1&\cdots &\bm{m}_p
\end{pmatrix}
\prod_{j=1}^{p} \partial^{\bm{m}_j}u
\notag\\
&=
{pu^{p-1}\partial^{\bm{\nu}}u
+
\sum_{n=0}^{p-2}
\frac{p!}{n!(p-n)!}
\sum_{(\bm{m}_i)_{i=1}^{p}\in B(n,\pi) }
\begin{pmatrix}
&\bm{\nu}
\\
\bm{m}_1&\cdots &\bm{m}_{p}
\end{pmatrix}
\prod_{j=1}^{p} \partial^{\bm{m}_j}u}
\notag\\
&=
pu^{p-1}\partial^{\bm{\nu}}u
+
\sum_{n=0}^{p-2}
{
\begin{pmatrix}
p\\
n
\end{pmatrix}
}
u^n\sum_{(\bm{m}_i)_{i=1}^{p-n}\in A_{p-n}' } 
\begin{pmatrix}
&\bm{\nu}
\\
\bm{m}_1&\cdots &\bm{m}_{p-n}
\end{pmatrix}
\prod_{j=1}^{p-n} \partial^{\bm{m}_j}u,
\end{align}
where, again, the case of $n=p$ is zero as $A_0'=\emptyset$, and the case of $n=p-1$ gives the first term $pu^{p-1}\partial^{\bm{\nu}}u$. {In \eqref{last}, the second equality is by isolating full derivative of order $\bm{\nu}$, the third equality is by the disjoint decomposition of $A_p$ in \eqref{classification}, and the fourth equality is by \eqref{symmetric}.}
}
\end{proof}
\bigskip

{The lemma above can be seen as a variant of Fa\'a di Bruno's formula, but the representation we obtained above is for special index set. The reason for the lemma above is that the obtained representation is convenient for our main theorem, and it is hard to deduce the desired representation from known Fa\'a di Bruno's formula directly. In order for the sanity check for \eqref{gprep}, a simple example is provided in Appendix \ref{appendixc}.} 
\\

{
For our convenience, we will be using the following notation for the rest of this paper:
\begin{align}
\label{gamma}
\Gamma(p,\bm{\nu})=\sum_{n=0}^{p-2}
{\begin{pmatrix}
p\\
n
\end{pmatrix}}
u^n\gamma(p-n,\bm{\nu}),
\end{align}
where
\begin{align}
\gamma(n,\bm{\nu}):=\sum_{\substack{\sum_{j=1}^n \bm{m}_j=\bm{\nu} \\ 
0\neq \bm{m}_j<\bm{\nu}}}
\begin{pmatrix}
&\bm{\nu}\\
\bm{m}_1&\cdots &\bm{m}_n
\end{pmatrix}
\prod_{j=1}^{n} \partial^{\bm{m}_i}u.
\end{align}
}

\subsection{Uniform bounds of smallest eigenvalues.}\label{uniform gap}In our analysis of estimating the mixed derivatives, the uniform boundedness of the smallest eigenvalue and the corresponding eigenfunction takes one of the crucial roles. In this subsection, we show the uniform boundedness of the ground eigenpairs of $\mathcal{O}$ and $\mathcal{T}$.

 When showing uniform bounds, a useful technique is to show that $\lambda(U)$, the image of the parameter space $U$ under the map $\lambda: U \rightarrow \mathbb{R}$ is compact. However, even if we have the result that $\lambda : U\rightarrow \mathbb{R}$ is analytic, we cannot conclude that $\lambda(U)$ is compact because it is hard to see if $U$ is compact. For example, the compactness of $U$ depends on the topology $U$ is living in and we have not defined it. Fortunately, by Theorem 2.7 in \cite{Cohen2015}, it has been shown that $b(U)$, the image of $U$ under the map $b:U\rightarrow L^\infty$ is compact in $L^\infty$. In order to use this fact, we want to investigate the continuous dependence of the eigenvalue on the coefficient function $b$.
 
   Now, note that we can consider solving the problem \eqref{semi1} as mapping the coefficient function $b$ to the eigenpair $(u,\lambda)$. In other words, we can see $u=u(b)$ and $\lambda=\lambda(b)$. In this perspective, the following lemma shows the desired continuous dependence of the eigenvalue on the coefficient function $b$ using the implicit function theorem technique used in Theorem \ref{anal}.\\

\begin{lemma}
\label{comp}
 Suppose that $(u,\lambda)$ is the ground eigenpair of problem \eqref{semi1}. Then the eigenvalue, $\lambda: L^\infty\rightarrow \mathbb{C}$, continuously depends on the coefficient function $b$.
\end{lemma}
\bigskip 

\begin{proof}
We use a similar analysis with Theorem \ref{anal}.
Let $u$ be the ground state and let us define an operator $\mathcal{M}(c,\ell,w)=-\nabla\cdot(a\nabla)(u+w)+c(u+w)+\eta (u+w)^p-\ell (u+w)$ on $L^\infty\times \mathbb{C}\times E$ where $E\subset H_0^1$ will be specified later. It is clear that $\mathcal{M}$ is analytic. Let $b\in L^\infty$ be given. Then, by the previous section, we have a unique ground state $(\lambda, u)$ so $\mathcal{M}(b,\lambda,0)=0$. Then observe that
\begin{align}
&D_{(\ell,w)}\mathcal{M}(b,\lambda,0)(\mu,f)\notag\\
&=\lim_{h\rightarrow \infty}\frac{\mathcal{M}(b,\lambda+h\mu,hf)-\mathcal{M}(b,\lambda,0)}{h}\notag\\
&=\lim_{h\rightarrow \infty}\frac{-\nabla\cdot(a\nabla)(u+hf)+b(u+hf)+\eta (u+hf)^p-(\lambda+h\mu) (u+hf)}{h}\notag\\
&=\left( -\nabla\cdot(a\nabla)+b +p\eta u^{p-1}-\lambda  \right)f-\mu u.
\end{align}
We note that $\lambda$ is not an eigenvalue of the linear operator $-\nabla\cdot(a\nabla)+b +p\eta u^{p-1}$ with similar argument used in Lemma \ref{anal}. Thus,  $\mathcal{J}:=\left( -\nabla\cdot(a\nabla)+b +p\eta u^{p-1}-\lambda  \right):H_0^1\rightarrow H^{-1}$ is an isomorphism. Thus, there is $w\in H_0^1$ such that $\mathcal{J}w=u$. Now, defined $E=span(w)^{\perp}$. As $\mathcal{J}$ is linear, $u$ is not in the image of $E$ under $\mathcal{J}$ and so  $D_{(\ell,w)}\mathcal{M}(b,\lambda,0):\mathbb{R}\times E\rightarrow H^{-1}$  is an isomorphism. Then, by the implicit function theorem, there is an open neighborhood $N_b\subset L^\infty$ of the coefficient $b$ such that there are continuous functions $\lambda: N_b\rightarrow \mathbb{R} $ and $f: N_b\rightarrow E$ such that $\mathcal{M}(V,\lambda(V),f(V))=0$ for all $V\in N$ with $\lambda(b)=\lambda$ and $f(b)=0$. 

In conclusion, we have obtained that, for each $b\in L^\infty$, there exists an open neighborhood $N_b$ such that the ground state is a continuous function in $N_b$. As the ground state is unique, with different $b_1,b_2\in L^\infty$, the ground state is well defined on $N_{b_1}\cap N_{b_2}$ if it is not empty. Thus, we can extend our ground state to be continuous over $L^\infty$.
\end{proof}
\bigskip

 With Lemma \ref{comp} shown above, we can claim the uniform boundedness of ground eigenpair as below. \\

\begin{lemma}
\label{ubound}
Suppose that $(u(\bm{y}),\lambda(\bm{y}))$ is the ground eigenpair of \eqref{semi}. Then we have $\lambda \in L^\infty(U,\mathbb{R})$, $u\in L^\infty(U,L^2)$ and $u\in L^\infty(U,H_0^1)$. Furthermore, $\lambda$ is uniformly strictly positive.
\end{lemma}
\bigskip 

\begin{proof}
By the Lemma 2.7 in \cite{Cohen2015}, the image of the parameter space $U$ under the map $b$, $b(U)$, is a compact set in $L^\infty$. Thus, by the Lemma \ref{comp}, we deduce that $\lambda(b(U))$ is compact because the image of a compact set under a continuous function is compact. It implies that $\lambda\in L^\infty(U,\mathbb{R})$. Furthermore, we know from the normalization condition that $u\in L^\infty(U, L^2)$. Then, these imply that 
\begin{align}
0\leq a_{min}\int|\nabla u(\bm{y})|^2 <  \left<\mathcal{O}u(\bm{y}),u(\bm{y})  \right> =   \lambda(\bm{y})
\end{align}
which tells us that $u\in L^\infty(U,H_0^1)$. In addition, the strict positivity of $u(\bm{y})$ implies the strict inequality, and it implies that the eigenvalue is strictly positive. Again, because $\lambda(U)=\lambda(b(U))$ is compact, there exists a constant $C>0$ such that $\lambda(\bm{y})\geq C>0$ and so $\lambda$ is uniformly away from $0$.
\end{proof}
\bigskip

Now, we denote $\overline{\lambda}$ and $\overline{u}$ to be the uniform upper bound for $\lambda(\bm{y})$ and $\|u(\bm{y})\|_{H_0^1}$. 
The uniform gap between the first and second smallest eigenvalues is crucial in the analysis of the linear case in \cite{alex2019}. In our nonlinear case, the uniform gap between the smallest eigenvalues of $\mathcal{T}$ and $\mathcal{O}$ shown below takes an important role.
\\

\begin{lemma}
\label{udiff} 
There exist a constant $C_T>0$ that is independent of parameters and satisfy the following inequality:
\begin{align}
\label{ct}
\lambda_T(\bm{y})-\lambda(\bm{y}) \geq C_{T}.
\end{align}
for all $\bm{y}\in U$.
\end{lemma}

\begin{proof}Observe that
\begin{align}
\lambda(\bm{y}) +(p-1)\eta\int  u(\bm{y})^{p-1} u_T(\bm{y})^2 
&\leq 
 \left<\mathcal{O}u_T(\bm{y}),u_T(\bm{y}) \right>+(p-1)\eta\int  u(\bm{y})^{p-1} u_T(\bm{y})^2\notag
\\& =
   \left< \mathcal{T}u_T(\bm{y}),u_T(\bm{y})  \right>\notag
 \\&   =
   \lambda_T(\bm{y}),
\end{align}
where the first inequality is because $\lambda(\bm{y})$ is the ground eigenvalue of $\mathcal{O}$. It implies that
\begin{align}
0<(p-1)\eta\int  u(\bm{y})^{p-1} u_T(\bm{y})^2 \leq \lambda_T(\bm{y})-\lambda(\bm{y}),
\end{align}
where the strict inequality follows from Lemma \ref{sp}. Because $\lambda_T$ and $\lambda$ are continuous functions in the potential function by Lemma \ref{comp}, thanks to compactness shown in \cite{Cohen2015}, $(\lambda_T-\lambda)(U)=(\lambda_T-\lambda)(b(U))$ is compact. Therefore, there exists $C_T>0$ such that  \eqref{ct} holds.
\end{proof}
\bigskip

\subsection{Bounding the mixed derivatives of the ground eigenpair.}\label{bounding eigenpair}In this section, we will estimate the bound for the mixed derivatives. In the analysis, we will use the falling factorial technique suggested by \cite{alexey} for the elliptic eigenvalue problem. For any $r\in \mathbb{R}$ and $n\in \mathbb{N}$, we define
\begin{align}
(r)_n:=
\begin{cases}
1, &n=0,\\
r(r-1)\cdots {(r-(n-1))}, &n\geq 1.
\end{cases}
\end{align}
And we will mainly use, for any $n\in \mathbb{N}$,
\begin{align}
\left[ \frac{1}{2} \right]_n:=\left| \left(\frac{1}{2} \right)_n\right|.
\end{align}
The result of \cite{alexey} shows that this tool helps us to remove $\varepsilon>0$ from the estimation $C(|\bm{\nu}|!)^{1+\varepsilon}\bm{\beta}^{\bm{\nu}}$  obtained for the linear eigenvalue problem in \cite{alex2019}. Two useful properties for our case would be
\begin{align}
\label{fall_ineq}
\left[\frac{1}{2} \right]_{n} \leq n! \leq 2^{n+1}\left[\frac{1}{2} \right]_{n},
\end{align}
and
\begin{align}
\label{fall_ineq1}
\sum_{i=1}^{n-1}
\begin{pmatrix}
n\\i
\end{pmatrix}
\left[ \frac{1}{2} \right]_{i}
\left[\frac{1}{2}
\right]_{n-i}\leq 2\left[\frac{1}{2}  \right]_{n}.
\end{align}
For more information about this tool, see section 2.2(especially Lemma 2.3) in \cite{alexey}.\\

\begin{lemma}
\label{ggcombi}
For given multi-index $\bm{\nu}\in \mathcal{F}$  and a positive integer $p>0$,
\begin{align}
\label{gg}
\sum_{\substack{\sum_{j=1}^p \bm{m}_j=\bm{\nu}\\ 
0\neq \bm{m}_j<\bm{\nu}}}
\begin{pmatrix}
&\bm{\nu}\\
\bm{m}_1&\cdots &\bm{m}_p
\end{pmatrix}
\prod_{j=1}^{p} \left[ \frac{1}{2} \right]_{|\bm{m}_j|}
\leq 2^{p-1} \left[ \frac{1}{2} \right]_{|\bm{\nu}|}.
\end{align}
\end{lemma}
\begin{proof}
Let a positive integer $p>0$ and a multi-index $\bm{\nu}$ be given. We first observe that we can express the left-hand side in \eqref{gg} by the nested sum as follows:
\begin{align}
\label{combi1}
&\sum_{\substack{\sum_{j=1}^p \bm{m}_j=\bm{\nu} \\ 
0\neq \bm{m}_j<\bm{\nu}}}
\begin{pmatrix}
&\bm{\nu}\\
\bm{m}_1&\cdots &\bm{m}_p
\end{pmatrix}
\prod_{j=1}^{p} \left[ \frac{1}{2} \right]_{|\bm{m}_j|} 
\notag\\
&=
 \sum_{k_1=1}^{|\bm{\nu}|-1} \sum_{\substack{\sum_{j=1}^p \bm{m}_j=\bm{\nu} \\ 
0\neq \bm{m}_j<\bm{\nu} \\ |\bm{m}_1|=k_1}}
\begin{pmatrix}
&\bm{\nu}\\
\bm{m}_1&\cdots &\bm{m}_p
\end{pmatrix}
\prod_{j=1}^{p} \left[ \frac{1}{2} \right]_{|\bm{m}_j|}\notag\\
&=
\sum_{k_1=1}^{|\bm{\nu}|-1} \sum_{k_2=1}^{|\bm{\nu}|-k_1-1} \sum_{\substack{\sum_{j=1}^p \bm{m}_j=\bm{\nu} \\ 
0\neq \bm{m}_j<\bm{\nu} \\ |\bm{m}_1|=k_1, 
|\bm{m}_2|=k_2}}
\begin{pmatrix}
&\bm{\nu}\\
\bm{m}_1&\cdots &\bm{m}_p
\end{pmatrix}
\prod_{j=1}^{p} \left[ \frac{1}{2} \right]_{|\bm{m}_j|}\notag\\
&=\cdots\notag\\
&=
\sum_{k_1=1}^{|\bm{\nu}|-1}
 \sum_{k_2=1}^{|\bm{\nu}|-k_1-1}
 \cdots
 \sum_{k_{p-1}=1}^{|\bm{\nu}|-\sum_{i=1}^{p-2}k_i-1}
 \prod_{j=1}^{p} \left[ \frac{1}{2} \right]_{k_j}
  \sum_{\substack{\sum_{j=1}^p \bm{m}_j=\bm{\nu} \\ 
0\neq \bm{m}_j<\bm{\nu} \\ |\bm{m}_j|=k_j, 1\leq j\leq p-1}}
\begin{pmatrix}
&\bm{\nu}\\
\bm{m}_1&\cdots &\bm{m}_p
\end{pmatrix},
\end{align}
where we used a special notation  ${k}_p:=|\bm{\nu}|-\sum_{i=1}^{p-1}k_i$ in the product factor for our convenience of indexing.
And, by Theorem 2 in \cite{chu}, we have
\begin{align}
\label{combi2}
\sum_{\substack{\sum_{j=1}^p \bm{m}_j=\bm{\nu} \\ 
0\neq \bm{m}_j<\bm{\nu} \\ |\bm{m}_j|=k_j, 1\leq j\leq p-1}}
\begin{pmatrix}
&\bm{\nu}\\
\bm{m}_1&\cdots &\bm{m}_p
\end{pmatrix}= 
\begin{pmatrix}
&|\bm{\nu}|\\
k_1 &\cdots &|\bm{\nu}|-\sum_{i=1}^{p-1}k_i
\end{pmatrix}.
\end{align}
Then, using the notation for $k_p$ defined above, the right hand side of \eqref{combi2} can be further computed to 
\begin{align}
\label{combi3}
\begin{pmatrix}
&|\bm{\nu}|\notag\\
k_1 &\cdots &k_p
\end{pmatrix} 
&
= 
\frac{|\bm{\nu}|!}{k_1! \cdots k_p!}\\
&
=
\frac{|\bm{\nu}|!}{k_1!\cdots k_{p-2}!(k_{p-1}+k_p)!}\frac{(k_{p-1}+k_p)!}{k_{p-1}! k_p!}
\notag\\
&
=
\frac{|\bm{\nu}|!}{k_1!\cdots k_{p-2}!(k_{p-2}+k_{p-1}+k_p)!}
\frac{(k_{p-2}+k_{p-1}+k_p)!}{(k_{p-1}+k_p)!k_{p-2}!}
\frac{(k_{p-1}+k_p)!}{k_{p-1}! k_p!}
\notag\\
&=\cdots
\notag\\
&=
\prod_{j=1}^{p-1} 
\begin{pmatrix}
\sum_{i=j}^{p} k_i
\\
k_j
\end{pmatrix}
\end{align}
Recall that $|\bm{\nu}|-\sum_{i=1}^{p-2}k_i=k_p+k_{p-1}$, and, in general, we have $|\bm{\nu}|-\sum_{i=1}^{p-\ell}k_i=\sum_{j=p-\ell+1}^p k_j$.  Then, using this identity and combining \eqref{combi2} and \eqref{combi3}, we obtain, from \eqref{combi1}, 
\begin{align}
\label{combi4}
&\sum_{\substack{\sum_{j=1}^p \bm{m}_j=\bm{\nu} \\ 
0\neq \bm{m}_j<\bm{\nu}}}
\begin{pmatrix}
&\bm{\nu}\\
\bm{m}_1&\cdots &\bm{m}_p
\end{pmatrix}
\prod_{j=1}^{p} \left[ \frac{1}{2} \right]_{|\bm{m}_j|} 
\notag
\\
&
=
\sum_{k_1=1}^{|\bm{\nu}|-1}
\begin{pmatrix}
\sum_{i=1}^{p}k_i\\
k_1
\end{pmatrix}
\cdots
\sum_{k_{p-1}=1}^{|\bm{\nu}|-\sum_{i=1}^{p-2}k_i-1} 
\begin{pmatrix}
k_{p-1}+k_p\\
k_{p-1}
\end{pmatrix}
\prod_{j=1}^{p} \left[ \frac{1}{2} \right]_{k_j}
\notag\\
&
=
\sum_{k_1=1}^{|\bm{\nu}|-1}
\begin{pmatrix}
|\bm{\nu}|\\
k_1
\end{pmatrix}
\sum_{k_2=1}^{|\bm{\nu}|-k_1-1}
\begin{pmatrix}
|\bm{\nu}|-k_1\\
k_2
\end{pmatrix}
\cdots
\sum_{k_{p-1}=1}^{|\bm{\nu}|-\sum_{i=1}^{p-2}k_i-1} 
\begin{pmatrix}
|\bm{\nu}|-\sum_{i=1}^{p-2}k_i\\
k_{p-1}
\end{pmatrix}
\prod_{j=1}^{p} \left[ \frac{1}{2} \right]_{k_j}.
\end{align}
 Now, observe that, the last sum of \eqref{combi4} can be computed as
 \begin{align}
& \sum_{k_{p-1}=1}^{|\bm{\nu}|-\sum_{i=1}^{p-2}k_i-1} 
\begin{pmatrix}
|\bm{\nu}|-\sum_{i=1}^{p-2}k_i\\
k_{p-1}
\end{pmatrix}
\prod_{j=1}^{p} \left[ \frac{1}{2} \right]_{k_j}
\notag\\
&=
\prod_{j=1}^{p-2} \left[ \frac{1}{2} \right]_{k_j}
 \sum_{k_{p-1}=1}^{|\bm{\nu}|-\sum_{i=1}^{p-2}k_i-1} 
\begin{pmatrix}
|\bm{\nu}|-\sum_{i=1}^{p-2}k_i\\
k_{p-1}
\end{pmatrix}
 \left[ \frac{1}{2} \right]_{k_{p-1}}
 \left[ \frac{1}{2} \right]_{k_{p}}
\notag\\
&=
\prod_{j=1}^{p-2} \left[ \frac{1}{2} \right]_{k_j}
 \sum_{k_{p-1}=1}^{|\bm{\nu}|-\sum_{i=1}^{p-2}k_i-1} 
\begin{pmatrix}
|\bm{\nu}|-\sum_{i=1}^{p-2}k_i\\
k_{p-1}
\end{pmatrix}
 \left[ \frac{1}{2} \right]_{k_{p-1}}
 \left[ \frac{1}{2} \right]_{|\bm{\nu}|-\sum_{i=1}^{p-2}k_i - k_{p-1}  }
\notag \\
 &=
 \left(\prod_{j=1}^{p-2} \left[ \frac{1}{2} \right]_{k_j}\right)
 2  \left[ \frac{1}{2} \right]_{|\bm{\nu}|-\sum_{i=1}^{p-2}k_i},  
 \end{align}
where the second equality is by the definition of $k_p$ and the last equality is by  \eqref{fall_ineq1}. Similarly, we can repeatedly apply \eqref{fall_ineq1} to \eqref{combi4} until all summation disappears. Then we obtain the desired expression.
\end{proof}
\bigskip

 Now, we are ready to prove the main theorem.
\\

\begin{theorem}
\label{main}
Suppose that $(u(\bm{y}),\lambda(\bm{y}))$ is the ground eigenpair of \eqref{semi}. Then for any multi-index $\bm{\nu}\in\mathcal{F} $ and parameter $\textbf{y}\in U$, the mixed derivatives of the ground eigenpair satisfy the following bounds:
\begin{align}
| \partial^{\bm{\nu}}\lambda (\bm{y})  |&\leq \overline{\lambda}(|\bm{\nu}|!) \bm{\beta^\nu},\\
\left\|  \partial^{\bm{\nu}}u(\bm{y}) \right\|_{H_0^1} &\leq \overline{u} (|\bm{\nu}|!)\bm{\beta^\nu},
\end{align}
where $\bm{\beta}=(\beta_i)_{i=1}^{\infty}$ with $\beta_i=C_{\bm{\beta}}\|b_i\|_{L^\infty}$ and   $C_{\bm{\beta}}>0$ is finite but large enough real number so that 
\begin{align}
\label{cb1}
\begin{split}
&\max\left\{
{
\frac{\tilde{C}_\lambda}{{C_{\bm{\beta}}}\overline{\lambda}},
\frac{\tilde{C}_u}{{C_{\bm{\beta}}}}  ,
 \frac{ \tilde{C}_{u'}}{{C_{\bm{\beta}}}\overline{u}} ,\frac{\tilde{C}_N}{{C_{\bm{\beta}}}} }
 \right\}
 \leq
  1
   \end{split} 
\end{align}
where {$\tilde{C}_\lambda={2}\left( \frac{(p-1)\eta C_{\mathcal{A}}^p\overline{u}^p}{C_T}+1 \right)$, $\tilde{C}_u={2}{C_T}$,  $\tilde{C}_{u'}=\frac{2C_{\mathcal{A}}}{a_{min}}\left(   3+\frac{2(p-1)\eta C_{\mathcal{A}}^p\overline{u}^p+\overline{\lambda}}{C_T} \right)$} and 
\begin{align}
\label{cn}
 {\tilde{C}_N=\frac{1}{C_T}
 \left(
2 \tilde{C}_\lambda  
+
\frac{5}{2}
+
\frac{\eta C_{\mathcal{A}}^p}{\tilde{C}_u}  \left[ \sum_{n=0}^{p-2}
{\begin{pmatrix}
p\\
n
\end{pmatrix}}
\overline{u}^n \tilde{C}_{u'}^{p-n}
2^{p-n-1}\right]
+ 
\frac{p-1}{2\tilde{C}_u}\eta C_{\mathcal{A}}^{p+2} \overline{u}^p
2 \tilde{C}_{u'}^2
  \right).}
\end{align}
{ Note that, by the Sobolev embedding theorem, we have $H_0^1\hookrightarrow L^{2p}$ with for all $(d,p)\in \mathcal{A}$.
  For simplicity, we denote {$C_{\mathcal{A}}>0$} to be common constant such that {$\|u\|_{L^r}\leq C_{\mathcal{A}}\|u\|_{H_0^1}$} for any $u\in H_0^1$ and $r\in \{2,2p\}$(just choose biggest one among them).} For convenience, we use $\tilde{C}_\lambda = C_{\bm{\beta}} C_\lambda$, $\tilde{C}_u=C_{\bm{\beta}}C_u$ and $\tilde{C}_{u'}=C_{\bm{
\beta}}C_{u'}$. 
\end{theorem}
Note that $\tilde{C}_\lambda$, $\tilde{C}_u$ and $\tilde{C}_{u'}$ are independent of $C_{\bm{\beta}}$ by the definition.
\bigskip 

\begin{proof} For the simplicity of our argument, let us drop the notation for the dependence on $\bm{y}$ so that $u=u(\bm{y})$ and $\lambda= \lambda(\bm{y})$ {as it will be clear from the context.}  We will show our result by mathematical induction. Note that the case of $\bm{\nu}=0$ is clear. Thus, we only focus on the case of $\bm{\nu}\neq 0$.

First, we claim that 
\begin{align}
\label{rel}
| \partial^{\bm{\nu}}\lambda  |
&\leq C_\lambda \left[\frac{1}{2} \right]_{|\bm{\nu}|} \bm{\beta^\nu},
\\
\label{rel2}
\left\|\partial^{\bm{\nu}} u  \right\|_{L^2}
 &\leq  C_u \left[\frac{1}{2} \right]_{|\bm{\nu}|}  \bm{\beta^\nu},
 \\
\label{reh1}
\left\|  \partial^{\bm{\nu}}u \right\|_{H_0^1} 
&\leq C_{u'} \left[\frac{1}{2} \right]_{|\bm{\nu}|}  \bm{\beta^\nu}.
\end{align}
 We will prove this claim by mathematical induction. We will also use the notation from \eqref{o} and \eqref{t} accordingly.

We start by observing that, by Theorem \ref{anal}, we can take mixed derivatives of the equation:
\begin{align}
\partial^{\bm{\nu}} \left( -\nabla\cdot(a\nabla) u+b(\bm{y})u +\eta u^p  \right) = \partial^{\bm{\nu}}\left( \lambda u \right).
\end{align}
Taking derivatives inside the parentheses and playing some algebraic manipulations, we have
\begin{align}
\label{deri}
&(-\nabla\cdot(a\nabla) +b(\bm{y}) +p\eta u^{p-1}  )\partial^{\bm{\nu}} u  
+\sum_{i=1}^{\infty}\nu_i b_i \partial^{\bm{\nu}-\bm{e}_i} u
+\eta\Gamma(p,\bm{\nu})
\notag\\
&= \partial^{\bm{\nu}}\lambda u +\lambda \partial^{\bm{\nu}}u +\sum_{0\neq \bm{m}<\bm{\nu}}
\begin{pmatrix}
\bm{\nu}\\
\bm{m}
\end{pmatrix}
\partial^{\bm{m}}\lambda \partial^{\bm{\nu}-\bm{m}} u.
\end{align}
For simplicity and convenience, we use the following notations:
\begin{align}
\label{lower}
&q=\sum_{i=1}^{\infty}\nu_i b_i \partial^{\bm{\nu}-\bm{e}_i}u, \\
&e=\sum_{0\neq \bm{m}<\bm{\nu}}
\begin{pmatrix}
\bm{\nu}\\
\bm{m}
\end{pmatrix}
\partial^{\bm{m}}\lambda \partial^{\bm{\nu}-\bm{m}}u.
\end{align}
Correspondingly, we define
\begin{align}
\label{cap}
&Q=\sum_{i=1}^{\infty} \nu_i \|b_i\|_{L^\infty} \left\| \partial^{\bm{\nu}-\bm{e}_i}u  \right\|_{L^2},\\
&E=\sum_{0\neq \bm{m}<\bm{\nu}}
\begin{pmatrix}
\bm{\nu}\\
\bm{m}
\end{pmatrix}
| \partial^{\bm{m}}\lambda |
\left\| \partial^{\bm{\nu}-\bm{m}} u\right\|_{L^2},\\
&D=\sum_{0\neq \bm{m}<\bm{\nu}} 
\begin{pmatrix}
\bm{\nu}\\
\bm{m}
\end{pmatrix}
\left\| \partial^{\bm{m}}u  \right\|_{H_0^1} \left\| \partial^{\bm{\nu}-\bm{m}}u \right\|_{H_0^1},
\\
&G = \sum_{n=0}^{p-2}
{\begin{pmatrix}
p\\
n
\end{pmatrix}}
\|u\|_{H_0^1}^n
\sum_{\substack{\sum_{j=1}^{p-n} \bm{m}_j=\bm{\nu} \\ 
0\neq \bm{m}_j<\bm{\nu}}}
\begin{pmatrix}
&\bm{\nu}\\
\bm{m}_1&\cdots &\bm{m}_{p-n}
\end{pmatrix}
\prod_{j=1}^{p-n} \|\partial^{\bm{m}_i}u\|_{H_0^1}.
\end{align}
To be specific, note that $q\leq Q$, $e\leq E$, and $\Gamma(p,\bm{\nu})\leq G$. $D$ will appear in \eqref{start}. Now, we start by estimating the first derivatives. Observe that, from \eqref{deri}, when $\bm{\nu}=\bm{e}_i$ for some $i\geq 1$,
\begin{equation}
\label{deri1}
(\mathcal{T}-\lambda)\partial^{\bm{e_i}} u  + b_iu = \partial^{\bm{e_i}}\lambda u 
\end{equation}
By multiplying $\partial^{\bm{e}_i}u$ on both sides and integrating them, we obtain
\begin{align}
\label{edit1}
\left< (\mathcal{T}-\lambda)\partial^{\bm{e}_i}u,\partial^{\bm{e}_i}u  \right> =-\left< b_i u,\partial^{\bm{e}_i}u \right>,
\end{align}
where we used the following property:
\begin{align}
0=\partial^{e_i}\left( \int u^2 \right) =2\int u\partial^{e_i} u=2\left<u,\partial^{e_i}u \right>. 
\end{align}
Then, it is clear that 
\begin{align}
\label{reason1}
(\lambda_T-\lambda) \|\partial^{\bm{e}_i}u\|_{L^2}^2\leq \left\|b_i \right\|_{L^\infty} \left\| \partial^{e_i}u\right\|_{L^2} = \frac{\beta_i}{C_{\bm{\beta}}}   \left\| \partial^{\bm{e}_i}u\right\|_{L^2}  
\end{align}
which implies that
\begin{align}
\label{ul21}
 \left\| \partial^{\bm{e}_i}u\right\|_{L^2} \leq \frac{\beta_i}{C_TC_{\bm{\beta}}}=C_u \left[\frac{1}{2} \right]_{1} \bm{\beta}^{\bm{e}_i},
\end{align}
where $C_T$ is from Lemma \ref{udiff}. Now, multiplying $u$ on both sides of \eqref{deri1} and {using the fact that $\mathcal{O}$ is a symmetric operator,} observe that
\begin{align}
\partial^{\bm{e}_i}\lambda 
= \left< \partial^{\bm{e}_i}u ,(\mathcal{O}-\lambda)u\right> +\left< b_i u,u \right> +\left< (p-1)\eta u^p,\partial^{\bm{e}_i}u\right>.
\end{align}
Since $u$ is the eigenfunction of $\mathcal{O}$ corresponding to the eigenvalue $\lambda$,  by H\"older's inequality, we obtain
\begin{align}
|\partial^{\bm{e}_i}\lambda|\leq \frac{\beta_i}{C_{\bm{\beta}}} + (p-1)\eta  \|u\|_{L^{2p}}^p \|\partial^{\bm{e}_i}u\|_{L^2}.
\end{align}
Thus,  by using \eqref{ul21},
\begin{align}
\label{l1bound}
|\partial^{\bm{e}_i}\lambda|\leq \frac{\beta_i}{C_{\bm{\beta}}} + (p-1)\eta C^p\overline{u}^p \frac{\beta_i}{C_T C_{\bm{\beta}}} = \frac{1}{C_{\bm{\beta}}} \left( 1+\frac{(p-1)\eta C^p\overline{u}^p}{C_T} \right)\beta_i= C_\lambda \left[\frac{1}{2} \right]_{1} \bm{\beta}^{\bm{e}_i}.
\end{align}
 From \eqref{edit1}, note that
 \begin{align}
 a_{min}\int |\nabla \partial^{\bm{e}_i}u|^2 
 &\leq \left< \partial^{\bm{e}_i}u, \mathcal{T} \partial^{\bm{e}_i}u\right>  \notag\\
 &=\lambda \left<\partial^{\bm{e}_i}u,\partial^{\bm{e}_i}u\right>
 -\left<b_iu,\partial^{\bm{e}_i}u  \right>  ,
 \end{align}
which implies, by \eqref{l1bound}, 
\begin{align}
\|\partial^{\bm{e}_i}u\|_{H_0^1}^2 
&\leq \frac{1}{a_{min}}\left( \overline{\lambda} \|\partial^{\bm{e}_i}u\|_{L^2}^2 +\|b_i\|_{L^\infty}\|\partial^{\bm{e}_i}u\|_{L^2}\right)
\notag\\
&\leq {\frac{C_{\mathcal{A}}\left\|\partial^{\bm{e}_i}u  \right\|_{H_0^1 }}{a_{min}}}
 \left( \overline{\lambda} \|\partial^{\bm{e}_i}u\|_{L^2}  +\|b_i\|_{L^\infty}  \right).
\end{align}
Then, by dividing both sides by $\|\partial^{e_i} u\|_{H_0^1}$ and using \eqref{ul21}, we get
\begin{align}
\label{uh11}
\|\partial^{\bm{e}_i}u\|_{H_0^1}
&\leq 
\frac{{C_{\mathcal{A}}}}{a_{min}}\left( \overline{\lambda} \frac{\beta_i}{C_TC_{\bm{\beta}}}+\frac{\beta_i}{C_{\bm{\beta}}}  \right) \notag\\
&= 
\frac{{C_{\mathcal{A}}}}{a_{min}C_{\bm{\beta}}}\left( \frac{\overline{\lambda}}{C_T}+1\right)\beta_i \notag\\
&\leq
 C_{u'} \left[\frac{1}{2} \right]_{1} \bm{\beta}^{\bm{e}_i}.
\end{align}
Now, let a multi-index $\bm{\nu}\in \mathcal{F}$ be given and  assume that \eqref{rel}, \eqref{rel2}, and \eqref{reh1} are true for all $\bm{\mu}<\bm{\nu}$. Under these assumptions, let us estimate $Q, G, D$, and $E$, which are from \eqref{cap}. Observe that, applying Lemma \ref{ggcombi}, we have the following estimations 
\begin{align}
\label{Q}
Q&=
\sum_{i=1}^{\infty} \nu_i \|b_i\|_{L^\infty} \left\| \partial^{\bm{\nu}-\bm{e}_i}u  \right\|_{L^2} 
\leq
 \sum_{i=1}^{\infty} \nu_i \frac{\beta_i}{C_{\bm{\beta}}} C_u\left[ \frac{1}{2} \right]_{|\bm{\nu}-\bm{e}_i|}\bm{\beta}^{\bm{\nu}-\bm{e}_i}
 =
\frac{C_u\bm{\beta}^{\bm{\nu}}}{C_{\bm{\beta}}}\left[ \frac{1}{2} \right]_{|\bm{\nu}|-1}\sum_{i=1}^{\infty} \nu_i \notag\\
&=\frac{C_u\bm{\beta}^{\bm{\nu}}}{C_{\bm{\beta}}}\left[ \frac{1}{2} \right]_{|\bm{\nu}|-1} |\bm{\nu}|
=
\frac{C_u\bm{\beta}^{\bm{\nu}}}{C_{\bm{\beta}}}\left[ \frac{1}{2} \right]_{|\bm{\nu}|-1} {\left( |\bm{\nu}|-1-\frac{1}{2} +\frac{3}{2} \right)}
\notag
\\ 
&\leq
\frac{C_u\bm{\beta}^{\bm{\nu}}}{C_{\bm{\beta}}}\left[\frac{1}{2} \right]_{|\bm{\nu}|-1}{\left(\left|\frac{1}{2}-\left(|\bm{\nu}|-1\right)\right|+\frac{3}{2}  \right)}
\notag\\
&\leq
{
\frac{C_u\bm{\beta}^{\bm{\nu}}}{C_{\bm{\beta}}}\left[\frac{1}{2} \right]_{|\bm{\nu}|-1}\left(\left|\frac{1}{2}-\left(|\bm{\nu}|-1\right)\right|+3 \left|\frac{1}{2}-\left(|\bm{\nu}|-1\right)\right| \right)}
\leq
 {\frac{4C_u}{C_{\bm{\beta}}}}\left[\frac{1}{2} \right]_{|\bm{\nu}|}\bm{\beta}^{\bm{\nu}} ,
\\
\label{G}
G &= \sum_{n=0}^{p-2}
{\begin{pmatrix}
p\\
n
\end{pmatrix}}
\|u\|_{H_0^1}^n
\sum_{\substack{\sum_{j=1}^{p-n} \bm{m}_j=\bm{\nu} \\ 
0\neq \bm{m}_j<\bm{\nu}}}
\begin{pmatrix}
&\bm{\nu}\\
\bm{m}_1&\cdots &\bm{m}_{p-n}
\end{pmatrix}
\prod_{j=1}^{p-n} \|\partial^{\bm{m}_i}u\|_{H_0^1}
\notag\\
&\leq  \sum_{n=0}^{p-2}
{\begin{pmatrix}
p\\
n
\end{pmatrix}}
\overline{u}^n
\sum_{\substack{\sum_{j=1}^{p-n} \bm{m}_j=\bm{\nu} \\ 
0\neq \bm{m}_j<\bm{\nu}}}
\begin{pmatrix}
&\bm{\nu}\\
\bm{m}_1&\cdots &\bm{m}_{p-n}
\end{pmatrix}
\prod_{j=1}^{p-n} C_{u'} \left[ \frac{1}{2}\right]_{|\bm{m}_j|}\bm{\beta}^{\bm{m}_j}
\notag\\
&\leq  \left( \sum_{n=0}^{p-2}
\textcolor{blue}{\begin{pmatrix}
p\\
n
\end{pmatrix}}
\overline{u}^n C_{u'}^{p-n}
2^{p-n-1}\right)
\left[\frac{1}{2} \right]_{|\bm{\nu}|}
\bm{\beta}^{\bm{\nu}},\\
\label{D}
D
&=\sum_{0\neq \bm{m}<\bm{\nu}} 
\begin{pmatrix}
\bm{\nu}\\
\bm{m}
\end{pmatrix}
\left\| \partial^{\bm{m}}u  \right\|_{H_0^1} \left\| \partial^{\bm{\nu}-\bm{m}}u \right\|_{H_0^1}
\notag\\
&\leq  \sum_{0\neq \bm{m}<\bm{\nu}} 
\begin{pmatrix}
\bm{\nu}\\
\bm{m}
\end{pmatrix}
C_{u'}^2
\left[ \frac{1}{2} \right]_{|\bm{\nu}-\bm{m}|}
\left[ \frac{1}{2} \right]_{|\bm{m}|}
\bm{\beta}^{\bm{\nu}}
\leq 2 C_{u'}^2 
\left[\frac{1}{2} \right]_{|\bm{\nu}|}
\bm{\beta}^{\bm{\nu}},\\
\label{E}
E&=\sum_{0\neq \bm{m}<\bm{\nu}}
\begin{pmatrix}
\bm{\nu}\\
\bm{m}
\end{pmatrix}
| \partial^{\bm{m}}\lambda |
\left\| \partial^{\bm{\nu}-\bm{m}} u\right\|_{L^2}
\notag\\
&\leq \sum_{0\neq \bm{m}<\bm{\nu}}
\begin{pmatrix}
\bm{\nu}\\
\bm{m}
\end{pmatrix}
C_\lambda C_u 
\left[ \frac{1}{2} \right]_{|\bm{\nu}-\bm{m}|}
\left[ \frac{1}{2} \right]_{|\bm{m}|}
 \bm{\beta}^{\bm{\nu}}
\leq 2 C_\lambda C_u 
\left[\frac{1}{2} \right]_{|\bm{\nu}|}
\bm{\beta}^{\bm{\nu}},
\end{align}
{where the last second inequality in \eqref{Q} is because $\left|\frac{1}{2}-\left(|\bm{\nu}|-1\right)\right|\geq 
\frac{1}{2}$ for all $\bm{\nu}\geq 2$.}
 Now, let's show the bound for $\|\partial^{\bm{\nu}} u\|_{L^2}$. Let $\{u_k\}_{k=1}^{\infty}$ be the sequence of eigenfunctions of $\mathcal{O}$ which is chosen to be an orthonormal basis for $L^2$, and it corresponds to the sequence of eigenvalues $0<\lambda_1 \leq \lambda_2 \leq \cdots $. Especially, note that $u_1=u$ and $\lambda_1=\lambda$. As $(u_k)_{k\geq 1}$ is a basis of $L^2$ space, $\partial^{\bm{\nu}}u$ can be represented by 
\begin{align}
\partial^{\nu}u = \sum_{k=1}^{\infty} \left< \partial^{\bm{\nu}}u, u_k  \right> u_k=\left< \partial^{\bm{\nu}}u, u  \right> u  +\sum_{k=2}^{\infty} \left< \partial^{\bm{\nu}}u, u_k  \right> u_k.
\end{align}
Let us denote the second term by $v$ for simplicity. In other words,
\begin{align}
\label{vortho}
v:=\sum_{k=2}^{\infty} \left< \partial^{\bm{\nu}}u, u_k  \right> u_k.
\end{align}
 The important fact about $v$ is that it is orthogonal to $u$. Then observe that
\begin{align}
\label{decompositionpu}
\left\| \partial^{\bm{\nu}} u  \right\|_{L^2} \leq |\left<\partial^{\bm{\nu}}u,u \right>| +\|v\|_{L^2}.
\end{align}
As for the first term, note that
\begin{align}
0=\partial^{\bm{\nu}}\left< u,u \right> = \sum_{\bm{m}\leq \bm{\nu}} 
\begin{pmatrix}
\bm{\nu}\\
\bm{m}
\end{pmatrix} \left<  \partial^{\bm{m}}u, \partial^{\bm{\nu}-\bm{m}}u\right>.
\end{align}
Thus, we have
\begin{align}
\left< u,\partial^{\bm{\nu}}u \right> =-\frac{1}{2}\sum_{0\neq \bm{m} < \bm{\nu}} 
\begin{pmatrix}
\bm{\nu}\\
\bm{m}
\end{pmatrix}
\left< \partial^{\bm{m}}u, \partial^{\bm{\nu}-\bm{m}}u\right>,
\end{align}
which implies 
\begin{align}
\label{start}
|\left<\partial^{\bm{\nu}}u,u \right>|  \leq \frac{1}{2}\sum_{0\neq \bm{m} < \bm{\nu}}
\begin{pmatrix}
\bm{\nu}\\
\bm{m}
\end{pmatrix}
 \| \partial^{\bm{m}}u\|_{L^2} \| \partial^{\bm{\nu}-\bm{m}}u\|_{L^2} =\frac{{C_{\mathcal{A}}}^2}{2}D.
\end{align}
Secondly, multiplying $v$ on both sides of \eqref{deri} and after some algebraic manipulations,
\begin{align}
\label{nuderi}
\left< (\mathcal{T}-\lambda)\partial^{\bm{\nu}}u , v \right> =\left<e-q-\eta\Gamma(p,\bm{\nu}),v \right>+\partial^{\bm{\nu}}\lambda \left<u,v \right>= \left<e-q-\eta\Gamma(p,\bm{\nu}),v \right>,
\end{align}
because  $\left<v,u \right>=0$. Observe that the left hand side of \eqref{nuderi} can be expanded as
 \begin{align}
 \label{lhs}
 \left< (\mathcal{T}-\lambda)\partial^{\bm{\nu}}u,v  \right>
 &=\left<(\mathcal{T}-\lambda)v,v \right>+\left<\partial^{\bm{\nu}}u,u \right>\left< (\mathcal{T}-\lambda)u,v  \right>
\notag \\ 
 &=\left<(\mathcal{T}-\lambda)v,v  \right> +(p-1)\eta \left<\partial^{\bm{\nu}}u,u \right> \left<u^p,v \right>, 
 \end{align}
where the second equality is due to the fact that $(\mathcal{\mathcal{O}}-\lambda)u=0$. Thus, combining \eqref{nuderi} and \eqref{lhs}, we have 
\begin{align}
\label{nuderi2}
\left<(\mathcal{T}-\lambda)v,v \right> = \left<e-q-\eta \Gamma(p,\bm{\nu}),v \right> -(p-1)\eta \left<\partial^{\bm{\nu}}u,u \right> \left<u^p,v \right>.
\end{align}
Using the fact that  $C_T\|v\|_{L^2}^2\leq \left<(\mathcal{T}-\lambda)v,v \right> $, H\"older inequality, and the Sobolev embedding theorem with \eqref{start}, we obtain
\begin{align}
C_T \|v\|_{L^2}^2\leq (E+Q+\eta {C_{\mathcal{A}}}^pG )\|v\|_{L^2}+ \frac{p-1}{2}
\eta {C_{\mathcal{A}}}^2D \|u\|_{L^{2p}}^p \|v\|_{L^2}.
\end{align}
By dividing $\|v\|_{L^2}$ from both sides by $C_T$, we have
\begin{align}
\label{vcom}
\|v\|_{L^2} \leq \frac{1}{C_T}\left( E+Q+  \eta {C_{\mathcal{A}}}^p G  +\frac{p-1}{2}\eta {C_{\mathcal{A}}}^{p+2}  \overline{u}^pD \right).
\end{align}
Then, combining \eqref{decompositionpu},  \eqref{start} and \eqref{vcom}, we see 
\begin{align}
\label{l2first}
\left\| \partial^{\bm{\nu}}u \right\|_{L^2} 
&\leq 
 \frac{1}{C_T}\left(  E+Q+  \eta {C_{\mathcal{A}}}^p G  +\frac{{C_{\mathcal{A}}}^2}{2}\left( (p-1)\eta  {C_{\mathcal{A}}}^p\overline{u}^p +{C_T}\right)D \right).
\end{align}
Then using the estimates \eqref{Q}-\eqref{E}, one can easily see the right-hand side of \eqref{l2first} is upper bounded by
\begin{align}
\label{l2pro}
& \frac{C_u }{C_{\bm{\beta}}C_T}
 \left(
2 \tilde{C}_\lambda  
+
{4}
+
{\frac{\eta }{C_{\mathcal{A}}}^p}{\tilde{C}_u}  \left[ \sum_{n=0}^{p-2}
{
\begin{pmatrix}
p\\
n
\end{pmatrix}
}
\overline{u}^n \tilde{C}_{u'}^{p-n}
2^{p-n-1}\right]
+ 
\frac{p-1}{\tilde{C}_u}\eta  {C_{\mathcal{A}}}^{p+2}\overline{u}^p
 \tilde{C}_{u'}^2
  \right) 
  \left[\frac{1}{2} \right]_{|\bm{\nu}|}\bm{\beta}^{\bm{\nu}},
\end{align}
where the tilde nations, $\tilde{C}_\lambda$, $\tilde{C}_u$, and $\tilde{C}_{u'}$ notate $C_{\bm{\beta}}$ independent constants.  Then, by \eqref{cn} and  \eqref{cb1}, we have 
\begin{align}
\label{l2}
\|\partial^{\bm{\nu}}u\|_{L^2} \leq C_N C_u\left[\frac{1}{2} \right]_{|\bm{\nu}|}\bm{\beta}^{\bm{\nu}} \leq C_u  \left[\frac{1}{2} \right]_{|\bm{\nu}|}\bm{\beta}^{\bm{\nu}}. 
\end{align}
Observe that, using the argument above, we have
\begin{align}
\label{usefulpro}
V:=C_TC_N C_u\left[\frac{1}{2} \right]_{|\bm{\nu}|}\bm{\beta}^{\bm{\nu}}\geq \left(  E+Q+  \eta {C_{\mathcal{A}}}^p G  +\frac{{C_{\mathcal{A}}}^2}{2}\left( (p-1)\eta {C_{\mathcal{A}}}^p\overline{u}^p +{C_T}\right)D \right) ,
\end{align}
which is useful for our next estimations. Now, from \eqref{deri}, with some algebraic manipulations, we obtain
\begin{align}
\label{lnu}
\partial^{\bm{\nu}} \lambda u =
\mathcal{T}\partial^{\bm{\nu}} u  
+q
+\eta \Gamma(p,\bm{\nu})-e -\lambda \partial^{\bm{\nu}}u.
\end{align}
Observe that, by making inner product of \eqref{lnu} and $u$, using the fact that $u$ is the eigenfunction of $\mathcal{O}$ with the eigenvalue $\lambda$ and using Cauchy-Schwarz inequality,
\begin{align}
\label{see1}
\partial^{\bm{\nu}} \lambda
&= \left<  (\mathcal{O}-\lambda)\partial^{\bm{\nu}}u,u  \right>
+
(p-1)\eta\left<\partial^{\bm{\nu}}u,u^p \right>
+
\left<q+\eta \Gamma(p,\bm{\nu})-e,u \right>\notag\\
&= \left<  (\mathcal{O}-\lambda)u,\partial^{\bm{\nu}}u  \right>
+
(p-1)\eta\left<\partial^{\bm{\nu}}u,u^p \right>
+
\left<q+\eta \Gamma(p,\bm{\nu})-e,u \right>\notag\\
&\leq (p-1)\eta \left\| u \right\|_{L^{2p}}^p \left\|\partial^{\bm{\nu}}u \right\|_{L^2}+E+Q+\eta {C_{\mathcal{A}}}^pG,
\end{align}
where we used the fact that $(\mathcal{O}-\lambda)$ is symmetric operator in the second equality.
Then, use the result of \eqref{l2} and the definition of $C_\lambda$ to have
\begin{align}
\partial^{\bm{\nu}} \lambda
&\leq 
(p-1)\eta{C_{\mathcal{A}}}^p\overline{u}^p \frac{1}{C_T}\left(  E+Q+  \eta {C_{\mathcal{A}}}^p G  +\frac{{C_{\mathcal{A}}}^2}{2}\left( (p-1)\eta   {C_{\mathcal{A}}}^p\overline{u}^p +{C_T}\right)D \right)
\notag\\
&+E+Q+\eta {C_{\mathcal{A}}}^p G \notag\\
&\leq
 \left( 1+\frac{(p-1)\eta {C_{\mathcal{A}}}^p\overline{u}^p}{C_T} \right)\left(E+Q+  \eta {C_{\mathcal{A}}}^p G  +\frac{{C_{\mathcal{A}}}^2}{2}\left( (p-1)\eta   {C_{\mathcal{A}}}^p\overline{u}^p +{C_T}\right)D \right)\notag\\
&\leq  \frac{C_{\bm{\beta }}{C}_\lambda}{2}  
 \left(  E+Q+  \eta {C_{\mathcal{A}}}^p G  +\frac{{C_{\mathcal{A}}}^2}{2}\left( (p-1)\eta   {C_{\mathcal{A}}}^p\overline{u}^p +{C_T}\right)D \right).
 \end{align}
 Lastly, using the estimated \eqref{usefulpro} and using the definition of $C_u$, we obtain
\begin{align}
\label{lambda}
\partial^{\bm{\nu}} \lambda
\leq \frac{C_{\bm{\beta}}C_\lambda}{2}C_TC_N C_u\left[\frac{1}{2} \right]_{|\bm{\nu}|}\bm{\beta}^{\bm{\nu}}
= C_N C_\lambda \left[\frac{1}{2} \right]_{|\bm{\nu}|}\bm{\beta}^{\bm{\nu}}\leq C_\lambda \left[\frac{1}{2} \right]_{|\bm{\nu}|}\bm{\beta}^{\bm{\nu}}.
\end{align}
 Recall that the last inequality is by \eqref{cb1}. Again, the tilde notations mean $C_{\bm{\beta}}$ independent version of corresponding constants.  Lastly, from \eqref{deri}, observe that
\begin{align}
\|\partial^{\bm{\nu}}u\|_{H_0^1}^2
&\leq \frac{1}{a_{min}}\left<-\nabla\cdot(a\nabla)\partial^{\bm{\nu}} u,\partial^{\bm{\nu}}u\right>\notag\\
& \leq  \frac{1}{a_{min}}\left<\mathcal{T}\partial^{\bm{\nu}}u,\partial^{\bm{\nu}}u \right>\notag\\
& = \frac{1}{a_{min}}\left( \left< e-q-\eta \Gamma(p,\bm{\nu}),\partial^{\bm{\nu}}u \right> +\partial^{\bm{\nu}}\lambda\left<u,\partial^{\bm{\nu}}u \right>+\lambda \left< \partial^{\bm{\nu}}u,\partial^{\bm{\nu}}u \right>\right).
\end{align}
Then, by using Cauchy-Schwarz inequality and the Sobolev embedding,
\begin{align}
\label{see}
\|\partial^{\bm{\nu}}u\|_{H_0^1}\leq \frac{{C_{\mathcal{A}}}}{a_{min}}\left( E+Q+\eta {C_{\mathcal{A}}}^p G  +|\partial^{\bm{\nu}}\lambda| +\lambda \|\partial^{\bm{\nu}}u\|_{L^2}   \right).
\end{align}
Using the previous estimates \eqref{l2} and \eqref{lambda}, 
\begin{align}
\|\partial^{\bm{\nu}}u\|_{H_0^1}
& \leq
  \frac{{C_{\mathcal{A}}}}{a_{min}}
  \left(
   E+Q+\eta
   {C_{\mathcal{A}}}^p G   + \left(  \frac{ C_{\bm{\beta }}{C}_\lambda}{2} + \frac{\overline{\lambda}}{C_T} \right) 
 V \right)\notag\\
 &\leq 
 \frac{{C_{\mathcal{A}}}}{a_{min}}\left(1+C_{\bm{\beta }}{C}_\lambda + \frac{\overline{\lambda}}{C_T} \right) 
 V\notag\\
 &\leq 
 \frac{{C_{\mathcal{A}}}}{a_{min}}\left( 3+\frac{2(p-1)\eta {C_{\mathcal{A}}}^p\overline{u}^p+\overline{\lambda}}{C_T} \right)
V,
\end{align}
 where $V$ is defined in \eqref{usefulpro} {and we used the definition of $C_{\lambda}$ in the last inequality}. Then the definitions of {$V$}, $C_{u}$ and $C_{u'}$ give us
\begin{align}
\|\partial^{\bm{\nu}}u\|_{H_0^1}^2 
\leq \frac{C_{\bm{\beta}}C_{u'}}{2}C_TC_N C_u\left[\frac{1}{2} \right]_{|\bm{\nu}|}\bm{\beta}^{\bm{\nu}} = C_N C_{u'}\left[\frac{1}{2} \right]_{|\bm{\nu}|}\bm{\beta}^{\bm{\nu}}\leq C_{u'}\left[\frac{1}{2} \right]_{|\bm{\nu}|}\bm{\beta}^{\bm{\nu}},
\end{align}
where we used \eqref{cb1} at the last inequality. Therefore, we have completed the induction \eqref{rel}-\eqref{reh1}.  Then, lastly, observing that, for any natural number $n$,
\begin{align}
\left[\frac{1}{2} \right]_n\leq n! \leq 2^{n+1}\left[\frac{1}{2} \right]_n,
\end{align}
and using \eqref{cb1} once again, we have our desired result.
\end{proof}
\bigskip

 Compared to \cite{alex2019, alexey}, our argument has an additional step to estimate the $L^2$ norm of the ground state. In their proof, the $H_0^1$ norm is estimated directly by the same technique we used for $L^2$ norm estimation, so estimating $L^2$ norm was not needed. However, in our proof strategy, estimation of $L^2$ norm is unavoidable. In short, it is mainly because of the nonlinear term. More precisely, in \eqref{nuderi2}, because $v$ is a linear combination of eigenfunctions of operator $\mathcal{O}$, not $\mathcal{T}$, it is hard to see if the inequality $C \|v\|_{H_0^1}\leq \left< (\mathcal{T}-\lambda)v,v\right>$ needed for $H_0^1$ estimation in their argument is true. 

 Note that, in our setting, the diffusion coefficient $a$ does not depend on the parameters. In our strategy, having $a$ to be parametrized has an obstacle arising from the same reason stated in the previous paragraph. To be specific, first observe that if $a=a_0+\sum_{j\geq 1}y_j a_j$ as in \cite{alex2019},  then \eqref{deri} is adapted to
 \begin{align}
\label{aderi}
&(-\nabla\cdot(a(\bm{y})\nabla) +b(\bm{y}) +p\eta u^{p-1}  )\partial^{\bm{\nu}} u
-\sum_{i=1}^{\infty} \nu_i\nabla\cdot( a_i\nabla)\partial^{\bm{\nu}-\bm{e}_i}u   
+\sum_{i=1}^{\infty}\nu_i b_i \partial^{\bm{\nu}-\bm{e}_i} u
+\eta\Gamma(p,\bm{\nu})
\notag\\
&= \partial^{\bm{\nu}}\lambda u +\lambda \partial^{\bm{\nu}}u +\sum_{0\neq \bm{m}<\bm{\nu}}
\begin{pmatrix}
\bm{\nu}\\
\bm{m}
\end{pmatrix}
\partial^{\bm{m}}\lambda \partial^{\bm{\nu}-\bm{m}} u.
\end{align}
In the case of when $\bm{\nu}=\bm{e}_1$, we have
\begin{align}
(\mathcal{T}-\lambda)\partial^{\bm{e}_1}u =\partial^{\bm{e}_1}\lambda u -(-\nabla\cdot(a_1\nabla )+b_1)u.
\end{align}  
 In our strategy, we multiply $\partial^{\bm{e}_1}u$ on both sides and integrate them which gives us
 \begin{align}
 C_T \|\partial^{\bm{e}_1}u\|_{L^2}
 &\leq
  \left< (\mathcal{T}-\lambda)\partial^{\bm{e}_1}u,\partial^{\bm{e}_1}u \right>\\
 &=
 \partial^{\bm{e}_1}\lambda\left<\partial^{\bm{e}_1}u,u \right>+\left< a_1 \nabla u, \nabla \partial^{\bm{e}_1}u  \right>
-\left< b_1 u,\partial^{\bm{e}_1}u\right> .
 \end{align}
Now, note that bounding $\|\partial^{\bm{e}_1}u\|_{L^2}$ needs bounding $\left< a_1 \nabla u, \nabla \partial^{\bm{e}_1}u  \right>$. Similarly, for general multi-index $\bm{\nu}$, we need to bound $\left< a_i \nabla u, \nabla \partial^{\bm{\nu}}u  \right>$ in order to bound $\|\partial^{\bm{\nu}}u\|_{L^2}$. However, this implies that we need to bound $\|\partial^{\bm{\nu}}u\|_{H_0^1}$ in order to bound $\|\partial^{\bm{\nu}}u\|_{L^2}$.  As stated in the last paragraph, the bound of $\|\partial^{\bm{\nu}}u\|_{L^2}$ is needed before estimating $|\partial^{\bm{\nu}}\lambda|$ and $\|\partial^{\bm{\nu}}u\|_{H_0^1}$(see \eqref{see1} and \eqref{see}).   Thus, with our strategy, we could not parametrize the diffusion coefficient. On this other hands, in \cite{alexey, alex2019}, it was possible to parametrize the diffusion coefficient $a$ as the they could use orthogonal decomposition technique \eqref{decompositionpu} to bound $H_0^1$ norm directly which is not feasible in our case as stated in the previous paragraph.

 Besides the eigenpair, its energy has its own interest. Note that, different from the linear eigenvalue problem where the eigenvalue is the energy of the system, in our nonlinear case, the eigenvalue and the energy are different. With the results of Theorem \ref{main}, we can also estimate  the energy in \eqref{genergy}:
 \begin{align}
 \mathcal{E}(\bm{y} ):= \inf_{u\in H} E_{\bm{y}}[u],
\end{align}  
 where
 \begin{align}
 E_{\bm{y}} [u] : =\int a|\nabla u|^2 +\int b(\bm{y})|u|^2 +\frac{2\eta}{p+1}\int |u|^{p+1}.
 \end{align}
{ It is clear that $\mathcal{E}\in L^\infty(U,\mathbb{R})$ due to the expression \eqref{lande} and the fact that $u\in L^\infty(U,H_0^1)$ from Lemma \ref{ubound}.}  When estimating $\mathcal{E}$, we need further restriction than $\mathcal{A}$ because $(p+1)$th power appears. Thus, in this case, noting that $\mathcal{A}'\subset \mathcal{A}$, we assume  $(d,p)\in \mathcal{A}'$ defined in \eqref{dp2}.
 \bigskip
 
 \begin{corollary}
 \label{esen}
 Let $(d,p)\in \mathcal{A}'$ and let $\bm{\nu}\in \mathcal{F}$ be a multi-index. Then, for any $\bm{y}\in U$, we have
\begin{align}
\label{en}
|\partial_{\bm{y}}\mathcal{E}(\bm{y})|\leq \overline{\mathcal{E}} (|\bm{\nu}|!) \bm{\beta}^{\bm{\nu}},
\end{align}
where $\overline{\mathcal{E}}$ is an upper bound for $\mathcal{E}:U\rightarrow \mathbb{R}$ and $\bm{\beta}=(\beta_i)_{i=1}^{\infty}$ with $\beta_i=C_{\bm{\beta}}\|b_i\|_{L^\infty}$. {The constant}  $C_{\bm{\beta}}>0$ is from Theorem \ref{main} which will be increased if necessary so that it satisfies
\begin{align}
  \frac{1}{C_{\bm{\beta}}}
\left(
 \tilde{C}_{\lambda}
 +
 \tilde{C}_u C_{u}^{p}\eta
{C_{\mathcal{A}}}^{p+1} 2^{p}
\frac{p-1}{p+1}
 \right) 
 \leq \overline{\mathcal{E}},
\end{align}
where we follow the notations from Theorem \ref{main}.

 \end{corollary}
\bigskip 

\begin{proof}
 Let $\bm{\nu}\in \mathcal{F}$ be given. We use the following representation of $\mathcal{E}(\bm{y})$,
\begin{align}
\lambda(\bm{y})=\mathcal{E}(\bm{y})+\eta\frac{p-1}{p+1}\int |u(\bm{y})|^{p+1}.
\end{align}
 Recall that we have shown \eqref{rel}, \eqref{rel2}, and \eqref{reh1}  in the proof of the last theorem. Using it, we observe that, by the product rule and the Sobolev inequality,
\begin{align}
|\partial^{\bm{\nu}} \mathcal{E}(\bm{y})| 
&\leq 
|\partial^{\bm{\nu}} \lambda(\bm{y})|
+
\eta\frac{p-1}{p+1} 
\sum_{
\substack{
\sum_{i=1}^{p+1} \bm{m}_i=\bm{\nu}
}
}
\begin{pmatrix}
&\bm{\nu}\\
\bm{m}_1 &\cdots & \bm{m}_{p+1}
\end{pmatrix}
{C_{\mathcal{A}}}^{p+1} \prod_{j=1}^{p+1}
\|\partial^{\bm{m}_j}u(\bm{y})\|_{H_0^1}
\notag\\
&\leq C_{\lambda} 
\left[\frac{1}{2} \right]_{|\bm{\nu}|}
\bm{\beta}^{\bm{\nu}}
+
\eta {C_{\mathcal{A}}}^{p+1}\frac{p-1}{p+1}
\sum_{
\substack{
\sum_{i=1}^{p+1} \bm{m}_i=\bm{\nu}
}
}
\begin{pmatrix}
&\bm{\nu}\\
\bm{m}_1 &\cdots & \bm{m}_{p+1}
\end{pmatrix}
\prod_{j=1}^{p+1} C_{u'}
\left[\frac{1}{2} \right]_{|\bm{m}_j|}
\bm{\beta}^{\bm{m}_j}\notag\\
&=
\left(C_{\lambda}+C_{u}^{p+1}\eta
{C_{\mathcal{A}}}^{p+1} 2^{p}
\frac{p-1}{p+1}
 \right)\left[\frac{1}{2} \right]_{|\bm{\nu}|}
  \bm{\beta}^{\bm{\nu}}\notag\\
 &=\frac{1}{C_{\bm{\beta}}}\left(\tilde{C}_{\lambda}+\tilde{C}_u C_{u}^{p}\eta
{C_{\mathcal{A}}}^{p+1} 2^{p}
\frac{p-1}{p+1}
 \right) 
\left[\frac{1}{2} \right]_{|\bm{\nu}|} 
  \bm{\beta}^{\bm{\nu}}\notag\\
&\leq \overline{\mathcal{E}}\left[\frac{1}{2} \right]_{|\bm{\nu}|} \bm{\beta}^{\bm{\nu}},
\end{align}
where we also used the result of Lemma \ref{ggcombi} in the third line. Now, using $\left[ \frac{1}{2}\right]_{|\bm{\nu}|}\leq |\bm{\nu}|!$, we obtain the desired result.
\end{proof}

\section{Application to Uncertainty Quantification}\label{application}

In this section, we suggest a possible application of the estimation of mixed derivatives of the ground eigenpair and the energy we obtained from the last section. We will discuss the error analysis for approximating the expectations of those functions. Specifically, the quantities of our interest are
\begin{align}
\mathbb{E}_{\bm{y}}\left[ \lambda(\bm{y})  \right]=\int_{\left[-\frac{1}{2},\frac{1}{2} \right]^{\mathbb{N}}} \lambda(\bm{y})d\bm{y}=\lim_{s\rightarrow \infty} \int_{\left[-\frac{1}{2},\frac{1}{2} \right]^s} \lambda(y_1,y_2,\dots, y_s,0,\dots)dy_1 dy_2\cdots dy_s,
\end{align}
and
\begin{align}
\mathbb{E}_{\bm{y}}\left[ \mathcal{G}(u(\bm{y}))  \right]=\int_{\left[-\frac{1}{2},\frac{1}{2} \right]^{\mathbb{N}}} \mathcal{G}(u(\bm{y}))d\bm{y}=\lim_{s\rightarrow \infty} \int_{\left[-\frac{1}{2},\frac{1}{2} \right]^s} \mathcal{G}(u(y_1,y_2,\dots, y_s,0,\dots))dy_1 dy_2\cdots dy_s.
\end{align}
for any $\mathcal{G}\in H^{-1}$. This section and Appendix \ref{appendixb} are straight forward adaptation of \cite{alex2019, sch1}. However, in those papers, the theorems for the error analysis are stated and proved focusing on specific PDEs. In order to justify that those theorems can be applied to our case, this paper restates their theorems in a way of generalization and modify the proofs for the completeness. In this process, the theorems are restated in a way that any solutions of any PDEs satisfying certain estimation of mixed derivatives have same quasi-Monte Carlo error convergence rate.

\subsection{Quasi-Monte Carlo methods.}
\label{qmc background}  In order to estimate these quantities, we will use quasi-Monte Carlo(QMC) method. QMC method, like the Monte Carlo method, estimates finite dimensional expectation. For example, in Monte Carlo(MC) method, we estimate the expectation of a function $f:\Omega\rightarrow \mathbb{R}$ with respect to a given  probability measure $\mathbb{P}$ by 
\begin{align}
\int_\Omega f(X)d\mathbb{P}(X) = \frac{1}{N}\sum_{i=1}^{N} f(X_i),
\end{align}
where $\Omega \subset \mathbb{R}^s$,  for a positive integer $s$, is a compact subset and $X_i \sim \mathbb{P}$ are iid samples. In QMC, instead of the iid samples as in the MC method, it uses a lattice rule which is designed for a better convergence rate. There are many different kinds of lattice rules for QMC method, but, in this paper, we use a lattice rule called \textit{CBC generated randomly shifted rank 1 rule}.

 \textit{CBC generated randomly shifted rank 1 lattice rule} is constructed with a generating vector $\bm{z}\in \mathbb{Z}^s$ and with  a random shift $\bm{\Delta}\sim Unif([0,1]^s)$. Specifically, the points we use are $\left( \left\{\frac{k\bm{z}}{N} +\bm{\Delta} \right\}-\bm{\frac{1}{2}} \right)_{k=0}^{N-1}$ where $\{\}$ notation means taking the fraction part of it, i.e., $\{\frac{3}{2}\}=\frac{1}{2}$. And $\bm{\frac{1}{2}}=\left(\frac{1}{2},\dots,\frac{1}{2} \right)\in \mathbb{R}^s$. Thus, through QMC, after the truncation, we want to estimate 
 \begin{align}
 \mathbb{E}_{\bm{y}_s}[f_s(\bm{y_s})] : = \int_{\left[-\frac{1}{2},\frac{1}{2} \right]^s} f(\bm{y}_s)d\bm{y}_s,
 \end{align}
by
\begin{align}
Q_{N,s} f := \frac{1}{N}\sum_{k=0}^{N-1}f\left( \left\{\frac{k\bm{z}}{N}+\bm{\Delta}  \right\}-\bm{\frac{1}{2}} \right),
\end{align}
where $f$ in our case would be $\lambda_s$ or $\mathcal{G}(u_s)$ defined by $\lambda_s(\bm{y}_s):=\lambda(\bm{y}_s;\bm{0})$ and $\mathcal{G}(u_s)(\bm{y}_s):=\mathcal{G}(u)(\bm{y}_s;\bm{0})$ with $\bm{y}_s :=(y_1,\dots, y_s)$. In this QMC rule, $\bm{z}$ is constructed using the Fast component-by-component(CBC) algorithm. For more details of this lattice rule or other rules, see  \cite{quasi}. Other works that used the randomly shifted lattice rule for various elliptic PDE can be found, for example, in \cite{para8, sch1, kien2022, alex2019}.

 As QMC method is for estimating a finite-dimensional integral, it is necessary to truncate the infinite-dimensional integral to a finite-dimensional one and it causes the truncation error. Thus, our error analysis would consist of two parts: truncation error and QMC error.

\subsection{Dimension truncation error.}
\label{truncation error}
 We will first estimate the parametric dimension truncation error and then the finite-dimensional integral by the QMC method. For convenience, let us notate, for given function $f: U\rightarrow X$ with a given normed vector space $X$, for any $\bm{y}\in U$, $\bm{y}_s:= (y_1,y_2,\dots, y_s)$ and $f_s:\left[ -\frac{1}{2},\frac{1}{2} \right]^s\rightarrow X$ defined by $f_s(\bm{y}_s)=f(\bm{y}_s;\bm{0})$.  The following lemma is a simple generalization of Theorem 4.1 in \cite{alex2019}, and the proof is also a simple adaptation except for the inclusion of case $q=1$. In order to point out the generality, we leave the proof in Appendix \ref{appendixb} for completeness.
\\

\begin{lemma}
\label{gtrunc}
Let $f: U\rightarrow X$ be given where $X$ is a given normed vector space with norm $\|\cdot\|_X$. Suppose that $f$ is analytic for each restriction to a finite-dimensional subspace and that $\|\partial^{\bm{\nu}}f(\bm{y})\|_X\leq g(\bm{\nu})\bm{\beta}^{\bm{\nu}}$ where $\bm{\beta}\in \ell^q$ with $q\in (0,1]$ is a decreasing sequence. And, for any fixed $k\in \mathbb{N}$, let 
\begin{align}
C_k:= \max_{\substack{\bm{\nu}\in \mathcal{F}\\ |\bm{\nu}|\leq k  }} \frac{g(\bm{\nu})}{2^{|\bm{\nu}|}\bm{\nu}!}.
\end{align}
 If there exists $k\in \mathbb{N}$ such that $C_k<\infty$,  then there exist $C>0$ such that, for sufficiently large $s\in \mathbb{N}$, 
\begin{align}
\label{strong}
\|f(\bm{y})-f_s(\bm{y}_s)\|_X \leq C T_1(s)
\end{align} 
where
\begin{align}
T_1(s)=
\begin{cases}
s^{-\frac{1}{q}+1}, &q\in (0,1),\\
\sum_{j>s}\beta_j, &q=1.
\end{cases}
\end{align}
Further suppose that, $C_k<\infty$ for any $k\in \mathbb{N}$. 
Then, there exists $C'>0$ such that, for sufficiently large $s\in \mathbb{N}$ and for any $\mathcal{G}\in X^*$,
\begin{align}
\label{weak}
\left| \mathbb{E}_{\bm{y}} [\mathcal{G}(f)-\mathcal{G}(f_s)]  \right|\leq C' T_2(s),
\end{align}
where
\begin{align}
T_2(s)=
\begin{cases}
s^{-\frac{2}{q}+1}, &q\in (0,1),\\
\left(\sum_{j>s}\beta_j\right)^2, &q=1.
\end{cases}
\end{align}
 Here, the constants $C$ and $C'$ are independent of $\bm{y}$ and $s$.
\end{lemma}
In the lemma above, \eqref{strong} is called a strong truncation error, and \eqref{weak} is called a weak truncation error. Then, applying Lemma \ref{gtrunc} to our case, we obtain the following theorem.\\

\begin{theorem}
\label{trunc}
Let $(d,p)\in \mathcal{A}$ for $u$ and $\lambda$ and $(d,p)\in \mathcal{A}'$ for $\mathcal{E}$. And let $(\lambda(\bm{y}),u(\bm{y}))$ is the ground state of \eqref{semi} with the Assumption 1. And suppose that $\left(  \|b_j\|_{L^\infty}  \right)_{j\geq 1}\in \ell^q$ for some $q\in (0,1]$ and $T_1(s)$ and $T_2(s)$ defined as Lemma \ref{gtrunc}. Then, for any $\bm{y}\in U$, there exist constants $c_k>0$ for $1\leq k\leq 6$ such that  the strong truncation error is bounded by 
\begin{align}
\label{lambdastrong}
|\lambda(\bm{y})-\lambda_s(\bm{y}_s)|\leq c_1 T_1(s),\\
|\mathcal{E}(\bm{y})-\mathcal{E}_s(\bm{y}_s)|\leq c_2 T_1(s),\\
\|u(\bm{y})-u(\bm{y}_s;\bm{0})\|_{H_0^1}\leq c_3 T_1(s).
\end{align}
The weak truncation error is bounded by 
\begin{align}
|\mathbb{E}_{\bm{y}}[\lambda-\lambda_s]|\leq c_4 T_2(s),\\
|\mathbb{E}_{\bm{y}}[\mathcal{E}-\mathcal{E}_s]|\leq c_5 T_2(s),
\end{align}
and, for any $\mathcal{G}\in H^{-1}$,
\begin{align}
\label{uweak}
|\mathbb{E}_{\bm{y}}[\mathcal{G}(u)-\mathcal{G}(u_s)]|\leq c_6 T_2(s).
\end{align}
\end{theorem}
\begin{proof}
Observe that, from Theorem \ref{main} and Corollary \ref{esen}, $\lambda$, $\mathcal{E}$ and $u$ are separately complex analytic whose derivatives satisfy the bound with $g(\bm{\nu})=C(|\bm{\nu}|!)$ and with $\|\cdot\|_X$ replaced by $|\cdot|$ and $\|\cdot\|_{H_0^1}$. Also, we can rearrange $\bm{\beta}$ so that it is decreasing. Also, it is clear that $C_k<\infty$ for any positive integer $k$ with the function $g$ defined above. Therefore, \eqref{lambdastrong} - \eqref{uweak} hold as desired with $\mathcal{G}$ to be replaced with identity function in the case of eigenvalue $\lambda$ and the energy $\mathcal{E}$.
\end{proof}
\bigskip

\subsection{Quasi-Monte Carlo error.}\label{qmc error} We can also reformulate the QMC error analysis of Theorem 4.2 in \cite{alex2019} and Theorem 6.4 in \cite{sch1} as follows. \\

\begin{lemma}
\label{gen}
Let $s\in \mathbb{N}$ be given. Let $X$ be a Banach space and let $f:\left[-\frac{1}{2},\frac{1}{2} \right]^s\rightarrow X$ be an analytic. And suppose that, for some constant $C>0$ and $\bm{\beta}\in \ell^q$ with $q\in (0,1)$, $\|\partial^{\bm{\nu}}f(\bm{y})\|_X\leq C (|\bm{\nu}|!) \bm{\beta}^{\bm{\nu}}$ for all $\bm{\nu}\in \mathcal{F}$ and for some $\varepsilon\in (0,1)$. Then, for any prime number $N$ and for any $\mathcal{G}\in X^*$, a lattice rule generating vector $\bm{z}\in \mathbb{N}^s$ can be constructed by CBC algorithm such that for some $s$ independent constant $C_\alpha>0$,
\begin{align}
\sqrt{\mathbb{E}_{\bm{\Delta}}\left[ \left| \mathbb{E}_{\bm{y}}\left[\mathcal{G}\left(f\right) \right]-Q_{N,s}\mathcal{G}(f)   \right|^2 \right]}\leq C_\alpha N^{-\alpha},
\end{align}
where $\bm{\Delta}\in [0,1]^s$ is the random shift and, for arbitrary $\delta\in \left(0,\frac{1}{2} \right)$,
\begin{align}
\alpha=
\begin{cases}
1-\delta, &q\in \left(0,\frac{2}{3} \right],  \\
\frac{1}{q}-\frac{1}{2},& q\in \left(\frac{2}{3},1\right].
\end{cases}
\end{align}
where, in the case of $q=1$, we assume additional condition, $\sum_{j\geq 1}\beta_j<\sqrt{6}$.
 \end{lemma}
The proof of the lemma above is almost identical to Theorem 4.2 in \cite{alex2019} except for the inclusion of the case $q=1$. For completeness and to clarify the applicability to our general setting, we leave the proof in Appendix \ref{appendixb}. We can conclude the following theorem by combining Theorem \ref{main} and the lemma above.\\

\begin{theorem}
Let $(d,p)\in \mathcal{A}$ for $u$ and $\lambda$ and $(d,p)\in \mathcal{A}'$ for $\mathcal{E}$. Moreover, let $s\in \mathbb{N}$ be given. Let $(\lambda(\bm{y}),u(\bm{y}))$ is the ground state of \eqref{semi} with the Assumption 1. And suppose that $\left(  \|b_j\|_{L^\infty}  \right)_{j\geq 1}\in \ell^q$ for some $q\in (0,1)$. Let $N\in \mathbb{N}$ be prime, $\mathcal{G}\in H^{-1}$. Then the root-mean-square errors of  the CBC-generated randomly shifted lattice rule approximations of $\mathbb{E}_{\bm{y}}[\lambda_s]$ and $\mathbb{E}_{\bm{y}}[\mathcal{G}(u_s)]$ with the generating vector constructed as in Lemma \ref{gen}  are bounded by
\begin{align}
\sqrt{\mathbb{E}_{\bm{\Delta}}\left[
\left| \mathbb{E}_{\bm{y}}[\lambda_s]-Q_{N,s}\lambda_s
\right|^2  \right]}\leq C_{1,\alpha} N^{-\alpha},\\
\sqrt{\mathbb{E}_{\bm{\Delta}}\left[
\left| \mathbb{E}_{\bm{y}}[\mathcal{E}_s]-Q_{N,s}\mathcal{E}_s
\right|^2  \right]}\leq C_{2,\alpha} N^{-\alpha},
\end{align}
and
\begin{align}
\sqrt{\mathbb{E}_{\Delta}
\left[ \left| 
\mathbb{E}_{\bm{y}}[\mathcal{G}(u_s)]-Q_{N,s}\mathcal{G}(u_s) \right|^2
\right]
}\leq C_{3,\alpha}N^{-\alpha},
\end{align}
where $\bm{\Delta}\in [0,1]^s$ is a random shift, the generating vector $\bm{z}\in \mathbb{N}^s$ is constructed as in Lemma \ref{gen} and
\begin{align}
\alpha=
\begin{cases}
1-\delta, &q\in \left(0,\frac{2}{3} \right],\\
\frac{1}{q}-\frac{1}{2}, &q\in \left( \frac{2}{3},1 \right],
\end{cases}
\end{align}
for arbitrary $\delta\in \left(0,\frac{1}{2} \right)$ and for some $s$ independent constants $C_{1,\alpha},C_{2,\alpha},C_{3,\alpha}>0$. And, in the case of $q=1$, we additionally assumed that $\sum_{j\geq 1}\beta_j<\sqrt{6}$ with $\beta_j=C_{\beta}\|b_j\|_{L^\infty}$ as defined in Theorem \ref{main}.

\end{theorem}

\subsection{Total error.}
\label{total error} We discussed the dimension truncation and the QMC error in the last two subsections. In this section, we combine those results to obtain the total error bound for estimating $\mathbb{E}_{\bm{y}}[\lambda]$ and $\mathbb{E}_{\bm{y}}[\mathcal{G}(u)]$ for any $\mathcal{G}\in H^{-1}$. We first present a general case.\\

\begin{lemma}
\label{gtotal}
Let $f: U\rightarrow X$ be given where $X$ is a given normed vector space with norm $\|\cdot\|_X$. Suppose that $f$ is analytic for any of its restrictions to a finite-dimensional domain and that $\|\partial^{\bm{\nu}}f(\bm{y})\|_X\leq C(|\bm{\nu}|!)\bm{\beta}^{\bm{\nu}}$ for some constant $C>0$, for some $\varepsilon\in [0,1)$ and for some decreasing sequence $\bm{\beta}\in \ell^q$.   Then, for sufficiently large $s\in \mathbb{N}$ and for any prime number $N\in \mathbb{N}$, there is a generating vector $\bm{z}\in \mathbb{N}^s$ such that, for any $\mathcal{G}\in X^*$ , 
\begin{align}
\sqrt{
\mathbb{E}_{\bm{\Delta}}
\left[
\left|
\mathbb{E}_{\bm{y}}[\mathcal{G}(f)]-Q_{N,s}\mathcal{G}(f_s)
\right|^2
\right]
}
\leq 
C_\alpha \left( T(s)+N^{-\alpha} \right),
\end{align}
where  
\begin{align}
T(s)=
\begin{cases}
s^{-\frac{2}{q}+1}, &q\in (0,1),\\
\left(\sum_{j>s} \beta_j \right)^2, &q=1,
\end{cases}
\end{align}
and
\begin{align}
\alpha=
\begin{cases}
1-\delta, &q\in \left(0,\frac{2}{3} \right],\\
\frac{1}{q}-\frac{1}{2}, &q\in \left( \frac{2}{3},1 \right],
\end{cases}
\end{align}
for arbitrary $\delta\in \left(0,\frac{1}{2} \right)$ and for some $s$ independent constants $C_{\alpha}>0$. When $q=1$, we additionally assume that $\sum_{j\geq 1}\beta_j<\sqrt{6}$.
\end{lemma}
\begin{proof}
The result is obvious by Lemma \ref{gtrunc} and Lemma \ref{gen} with triangular inequality. Indeed,
\begin{align}
&\sqrt{
\mathbb{E}_{\bm{\Delta}}
\left[
\left|
\mathbb{E}_{\bm{y}}[\mathcal{G}(f)]-Q_{N,s}\mathcal{G}(f_s)
\right|^2
\right]
}\\
&\leq 
\sqrt{
\mathbb{E}_{\bm{\Delta}}
\left[
\left|
\mathbb{E}_{\bm{y}}[\mathcal{G}(f)-\mathcal{G}(f_s)]
\right|^2
\right]
}
+\sqrt{
\mathbb{E}_{\bm{\Delta}}
\left[
\left|
\mathbb{E}_{\bm{y}}[\mathcal{G}(f_s)]-Q_{N,s}\mathcal{G}(f_s)
\right|^2
\right]
},
\end{align}
and, on the right-hand side, the first and the second terms are bounded by Lemma \ref{gtrunc} and Lemma \ref{gen}, respectively. 
\end{proof}
\bigskip

The following result is just a simple application of Theorem \ref{main} to Lemma \ref{gtotal}.\\

\begin{theorem}
\label{final}
Let $(d,p)\in \mathcal{A}$ for our ground eigen-pair ($u, \lambda)$ of \eqref{semi} with Assumption 1 and let $(d,p)\in \mathcal{A}'$ for the energy $\mathcal{E}$.  Suppose that  $\left(  \|b_j\|_{L^\infty}  \right)_{j\geq 1}\in \ell^q$ for some $q\in (0,1)$,  $N\in \mathbb{N}$ is a prime and $\mathcal{G}\in H^{-1}$. Then the root-mean-square errors of  the CBC-generated randomly shifted rank 1 lattice rule approximations of $\mathbb{E}_{\bm{y}}[\lambda]$, $\mathbb{E}_{\bm{y}}\left[\mathcal{E}(\bm{y})\right]$ and $\mathbb{E}_{\bm{y}}[\mathcal{G}(u)]$ for any $\mathcal{G}\in H^{-1}$ with the corresponding generating vectors $\bm{z}\in \mathbb{N}^s$ constructed as in Lemma \ref{gen} are bounded by
\begin{align}
\sqrt{\mathbb{E}_{\Delta}\left[
\left| \mathbb{E}_{\bm{y}}[\lambda]-Q_{N,s}\lambda_s
\right|^2  \right]}\leq C_{1,\alpha}\left(T(s)+ N^{-\alpha}\right),\\
\sqrt{\mathbb{E}_{\Delta}\left[
\left| \mathbb{E}_{\bm{y}}[\mathcal{E}]-Q_{N,s}\mathcal{E}_s
\right|^2  \right]}\leq C_{2,\alpha}\left(T(s)+ N^{-\alpha}\right),
\end{align}
and
\begin{align}
\sqrt{\mathbb{E}_{\Delta}
\left[ \left| 
\mathbb{E}_{\bm{y}}[\mathcal{G}(u)]-Q_{N,s}\mathcal{G}(u_s) \right|^2
\right]
}\leq C_{3,\alpha}\left(T(s)+ N^{-\alpha}\right),
\end{align}
where $T(s)$ is defined in Theorem \ref{gtotal} and
\begin{align}
\alpha=
\begin{cases}
1-\delta, &q\in \left(0,\frac{2}{3} \right],\\
\frac{1}{q}-\frac{1}{2}, &q\in \left( \frac{2}{3},1 \right],
\end{cases}
\end{align}
for arbitrary $\delta\in \left(0,\frac{1}{2} \right)$ and for some $s$ independent constants $C_{1,\alpha},C_{2,\alpha},C_{3,\alpha}>0$.
\end{theorem}
\bigskip 

Note that the error convergence rate we obtained from Theorem \ref{final} is identical to that of the linear eigenvalue problem in \cite{alex2019} even with the nonlinearity. Furthermore, throughout this section, we showed that it is mainly due to the form of the estimation we obtained in Theorem \ref{main}. Similarly, we can expect that any solution to PDEs with stochastic coefficients with the same form of the upper bound estimation as in Theorem \ref{main} would satisfy the same QMC error convergence rate. For example, (4.2) of \cite{sch1} shows that an elliptic equation with parametric diffusion coefficients satisfies the same form of the estimation. Furthermore, by the result of \cite{alexey}, the parametric linear elliptic eigenvalue problems with more various kinds of parametric coefficients have the same form of the estimation which will allow same QMC error convergence rate.

\textbf{Acknowledgement. }The author thanks professor Yulong Lu(University of Massachusetts Amherst) for all the helpful discussions and support during this research. The author also thanks the NSF for the partial support via the award DMS-2107934.

\appendix
\section{The proof of Theorem \ref{exiuni}}
\label{appendixa}

The proof of Theorem \ref{exiuni} is a combination of the following lemmata. The proofs of the Lemmta \ref{SC}-\ref{sol} and the Corollary \ref{absol} are simple adaptations of the appendix A in \cite{lieb2000}. Because we have more various nonlinearities and we  use bounded physical domain instead of whole domain used in \cite{lieb2000}, we present the adapted proof here for the completeness.  Throughout this appendix, recall the definition of $E$ in \eqref{energy} and we will use $H=\left\{ v\in H_0^1 : \|v\|_{L^2}=1 \right\}$ as defined in Section \ref{preliminary}. Also, we focus on case of $p>1$ because $p=1$ is linear and we can simply use Lax-Milgram theorem. The proof of Theorem \ref{exiuni} will be presented after all necessary tools are given.
\\

\begin{lemma} 
\label{SC}   For $w\geq 0$, $\sqrt{w}\in H$, define $F(w)=E(\sqrt{w})$. Then $F(w)$ is strictly convex with respect to $w$.\\
\end{lemma}
\begin{proof}
First,  observe that
\begin{align}
F(w) = \int a|\nabla \sqrt{w}|^2  +\int b w+\frac{2\eta}{p+1}\int |\sqrt{w}|^{p+1},
\end{align}
and define $F_1(w)=\int a|\nabla \sqrt{w}|^2$, $F_2(w)=\int bw$ and $\frac{2\eta}{p+1}\int |\sqrt{w}|^{p+1}$. Note that $F_2(w)$ is linear so is convex in $w$ and the $F_3(w)$ is strictly convex in $w$ because $\frac{p+1}{2}>1$. Thus, we only need to show that $F_1(w)$ is convex. 

Let $w_1, w_2$ be given such that $\sqrt{w_1},\sqrt{w_2}\in H$ and  $w_1\neq w_2$. Let $\sqrt{w}:=\sqrt{{\alpha w_1+(1-\alpha)w_2}}$. We want to show that $\sqrt{w}\in H$ and check $F_1(w)\leq \alpha F_1(w_1)+(1-\alpha)F_1(w_2)$ for any $\alpha\in [0,1]$. First, we can easily observe that
\begin{align}
\int |\sqrt{w}|^2 = \alpha \int|\sqrt{w_1}|^2 +(1-\alpha)\int |\sqrt{w_2}|^2 =1.
\end{align}
 Moreover, observe that
\begin{align}
F_1(w)=\int a|\nabla \sqrt{w}|^2&
=\int a\left| \nabla \sqrt{\alpha w_1 +(1-\alpha)w_2}\right|^2 \notag\\
&= \int a \frac{(\alpha\nabla w_1 +(1-\alpha)\nabla w_2)^2}{4|w|}\notag\\
&=\int a\frac{(2\alpha \sqrt{w_1} (\nabla \sqrt{w_1})+2(1-\alpha) \sqrt{w_2}(\nabla \sqrt{w_2}))^2}{4|w|} \notag \\
&\leq \int a\frac{|w|(\alpha |\nabla \sqrt{w_1}|^2+(1-\alpha)|\nabla \sqrt{w_2}|^2)}{|w|}\notag\\
&=\alpha \int a|\nabla \sqrt{w_1}|^2 +(1-\alpha)\int a|\nabla \sqrt{w_2}|^2 \notag\\
& <\infty,
\end{align} 
where the first inequality is by Cauchy-Schwarz inequality which is $(a'b'+c'd')^2\leq (a'^2+c'^2)(b'^2+d'^2)$ for all $a',b',c',d'\in \mathbb{R}$. It implies that $w\in H$ and $F_1(w)$ are convex in $w$ as we desired.
 \end{proof}
\bigskip

\begin{lemma} 
\label{M} There exists  $u\in H$ satisfying \eqref{genergy} that is unique up to complex modulus $|\cdot|$.
\end{lemma}
\begin{proof}
Suppose that $(u_n)_{n\geq 1}\subset H$ be such that 
\begin{align}
\lim_{n\rightarrow \infty}E[u_n] = \mathcal{E}=\inf_{u\in H}E[u].
\end{align}
This implies that, there exists $C>0$ such that ${E[u_n]}\leq C$ for all $n\geq 1$. Note that each term in $E[u_n]$ is a power of variations of $H_0^1$, $L^2$, and $L^{p+1}$ norms, respectively.   Thus, $u_n$ is contained in a  weakly compact subset of those normed spaces. It implies a subsequence still denoted by  $(u_n)_{n\geq 1}$ and $u$ such that $u_n$ weakly converges to $u$.
Then, by the weakly lower semicontinuity of the norms,
\begin{align}
\mathcal{E}=\underset{n\rightarrow \infty}{\mathrm{liminf}} E[u_n]\geq E[u]\geq \inf_{v\in \textcolor{blue}{H}} E[v]=\mathcal{E},
\end{align} 
which implies that $E[u]=\mathcal{E}$.
 Now, we show that $\|u\|_{L^2}=1$. Note that $H_0^1$ is compactly embedded into $L^2$ by Rellich compactness theorem. Thus $(u_n)_{n\geq 1}$ contained in a bounded set in $H_0^1$ is contained in a compact subset of $L^2$. Therefore, there is a subsequence of $(u_n)_{n\geq 1}$ still denoted by $(u_n)_{n\geq 1}$ converges to $u$ in $L^2$. Therefore, 
\begin{align}
1=\lim_{n\rightarrow \infty} \left\| u_n \right\|_{L^2} = \left\|u  \right\|_{L^2},
\end{align}
Finally, if there exists two different solution $v_1$ and $v_2$ such that $|v_1|\neq |v_2|$, then by Lemma \ref{SC}, observe that, for given $\alpha \in (0,1)$,
\begin{align}
\mathcal{E} \leq E[\sqrt{\alpha |v_1|^2 + (1-\alpha)|v_2|^2 }]&=F[\alpha |v_1|^2 + (1-\alpha)|v_2|^2 ]\\
&<\alpha F[|v_1|^2]+ (1-\alpha)F[|v_2|^2]
\\
&=\alpha E[|v_1|]+(1-\alpha)E[|v_2|]
\\
& = \mathcal{E},
\end{align}
which is a contradiction. Therefore, $|v_1|=|v_2|$.
\end{proof}
\bigskip

\begin{lemma}
\label{sol}
$u$ is the minimizer of \eqref{genergy} if and only if $u$ satisfies \eqref{semi1} with $\lambda = \mathcal{E} +\eta\frac{p-1}{p+1}\int |u|^{p+1}. $
\end{lemma}  
\bigskip 

\begin{proof}
Suppose that $u\in H$ is the minimizer of \eqref{genergy} with minimum $\mathcal{E}$.  Let $f\in C_0^\infty$ (space of smooth functions with compact support in $\Omega$) be given. Then, we should have
\begin{align}
\frac{d}{d\varepsilon} \left[   E[u+\varepsilon f] -\lambda \left\| u+\varepsilon f  \right\|_{L^2}^2 \right] \bigg|_{\varepsilon=0} =0,
\end{align} 
where $\lambda$ is the Lagrange multiplier corresponding to the normalization restriction.

Observe that, for any real-valued $f\in C_0^\infty$, 

\begin{align}
\label{er}
&\frac{d}{d\varepsilon} \left[   E[u+\varepsilon f] -\lambda \left\| u+\varepsilon f  \right\|_{L^2}^2 \right] \bigg|_{\varepsilon=0} \notag\\ 
&=\lim_{\varepsilon\rightarrow 0} \frac{E[u+\varepsilon f]-\lambda \|u+\varepsilon f\|^2_{L^2} - E[u]+\lambda \|u\|^2_{L^2}}{\varepsilon}\notag\\
&= 2\int (-\nabla\cdot (a\nabla Re(u))+b Re(u) + \eta |u|^{p-1} Re(u)-\lambda Re(u))f\notag\\
&=0,
\end{align}
Similarly, apply $if$ instead of $f$ from last argument,  we obtain
\begin{align}
\label{ei}
2\int (-\nabla\cdot (a\nabla Im(u))+b Im(u) + \eta |u|^{p-1} Im(u)-\lambda Im(u))f=0,
\end{align}
By combining \eqref{er} and \eqref{ei}, we obtain
\begin{align}
\label{weakeq1}
\int (-\nabla\cdot(a\nabla u) +b u +\eta |u|^{p-1} u-\lambda u) f =0,  &&\forall f\in C_0^\infty,
\end{align}
where $f$ is complex-valued. Therefore, $u$ satisfied \eqref{semi1} in the sense of distribution. The converse statement can be proved just by reversing the argument above, and the exact value of $\lambda$ is clear. Indeed, observe that, if $u$ is the solution to \eqref{exiuni}, then using, by the denseness of $C_0^\infty$ in $H_0^1$, $f=u$ in \eqref{weakeq1}, we obtain,
\begin{align}
\label{lande}
\lambda 
&=
 \lambda \int |u|^2
 \notag\\
 & = \int |\nabla u|^2 +\int b|u|^2+\eta \int |u|^{p+1}
 \notag\\
 &=\mathcal{E}-\frac{2\eta}{p+1}\int |u|^{p+1}+\eta\int |u|^{p+1}
 \notag\\
 &=\mathcal{E}+\eta\frac{p-1}{p+1}\int |u|^{p+1},
\end{align}
as desired.
\end{proof}
\bigskip

\begin{corollary}
\label{absol}
If $u$ is a minimizer of \eqref{energy}, then $|u|$ is also a minimizer, and so is the solution to \eqref{semi1}. 
\end{corollary}
\bigskip 

\begin{proof}
Suppose that $u$ is the minimizer of \eqref{energy}, and observe that
\begin{align}
E[|u|]&=\int a|\nabla |u||^2 +\int b|u|^2 +\frac{2\eta}{p+1}\int |u|^{p+1}\notag\\
& = \int a\left| \frac{u}{|u|}\nabla u \right|^2 +\int b |u|^2 +\frac{2\eta}{p+1}\int |u|^{p+1} =E[u].
\end{align}

Therefore, $|u|$ is also a minimizer of \eqref{energy} so is the solution to \eqref{semi1} by Lemma \ref{sol}.
\end{proof}
\bigskip

\begin{lemma}
\label{regul} Let $(d,p)\in \mathcal{A}$ and  $u$ be the solution to \eqref{semi1}. Then $u\in H^2$ and so is continuous. 
\end{lemma}
\bigskip 

\begin{proof}
 Observe that $u$ is the solution to the following Dirichlet boundary inhomogeneous elliptic problem:
 \begin{align}
 -\nabla\cdot(a\nabla v) +bv=\lambda u-\eta|u|^{p-1} u.
 \end{align}
Then note that, because $u\in H_0^1\subset L^2$ and $H_0^1\hookrightarrow L^{2p}$, we have   $\lambda u -\eta|u|^{p-1} u\in L^2$  for any $(d,p)\in \mathcal{A}$. Thus, by the elliptic regularity (see \cite{evans}, p.336), $u\in H^2$ because $a\in C^1$ and $\Omega$ has $C^2$ boundary. Then, by Sobolev embedding theorem, $u$ is H\"older continuous with some exponent depending on $d$. Thus, $u$ is continuous.
\end{proof}
\bigskip

\begin{lemma}
\label{sp} Suppose that $(d,p)\in\mathcal{A}$ and $u$ is the unique nonnegative real solution to \eqref{semi1}. Then $u$ is bounded and strictly positive.
\end{lemma}
\bigskip 

\begin{proof}  Note that the $u$ is the solution to 
\begin{align}
\label{lambdaeq}
-\nabla\cdot(a\nabla) v +bv=(\lambda-\eta|u|^{p-1})u ,
\end{align}
With Dirichlet boundary condition. Let $\Lambda =\{ x\in \Omega : \lambda-\eta |u|^{p-1}< 0  \}$ and note that it is open. And observe that $u$ satisfies \eqref{lambdaeq} with positive right hand side on $\Lambda$ with the boundary condition $u=\left({\frac{\lambda}{\eta}}\right)^{\frac{1}{p-1}}$ on $\partial \Lambda$. Since $u\in H^2$ by {Lemma} \ref{regul}, applying the maximum principle(see Theorem 8.19 in \cite{trudinger}), we have $u=\left( \frac{\lambda}{\eta}\right)^{\frac{1}{p-1}}$ on $\Lambda$. This implies that $\Lambda$ is empty and so $u\leq \left( \frac{\lambda}{\eta}\right)^{\frac{1}{p-1}}$ on $\Omega$ which proves the boundedness. Now, note that $u$ is the solution to 
\begin{align}
-\nabla\cdot (a\nabla) v +(b+\eta |u|^2-\lambda)v=0,  
\end{align}
With Dirichlet boundary condition. Then, since $(b+\eta|u|^{p-1}-\lambda)\in L^\infty$, the strict positivity is clear form Harnack's inequality.(see section 8.8 of \cite{trudinger})
\end{proof}
\bigskip

 We conclude this appendix by presenting the proof of Theorem \ref{exiuni} combining the lemmata and a corollary above. First of all, the existence of a solution unique up to complex modulus is a consequence of Lemma \ref{M} and \ref{sol} where Lemma \ref{SC} is used to prove Lemma \ref{M}. Corollary \ref{absol} shows that we have unique non-negative real-valued solution.  Lastly, due to Lemma \ref{sp}, we observe that the non-negative real-valued solution is indeed strictly positive which proves the last statement of Theorem \ref{exiuni}.

\section{Proofs of Lemma \ref{gtrunc} and Lemma \ref{gen}}
\label{appendixb}
\begin{proof}\textit{(Proof of Lemma \ref{gtrunc})}
Throughout this proof, we set $C', C>0$ as varying constants that do not depend on $s$. First, for given integer $s>0$, let us denote $\mathcal{F}_s:=\{0\neq \bm{\nu}\in \mathcal{F}: \nu_j=0\text{ for all }j=1,2,\dots, s\}$ and $\mathcal{F}_{\ell,s}:=\{\bm{\nu}\in \mathcal{F}_s : |\bm{\nu}|=\ell \}$. Then, note that, by $k$th order Taylor series expansion of $f$ in variable $\bm{y}_{ \{ j>s \} }$ is
\begin{align}
\label{taylor}
&f(\bm{y})-f(\bm{y}_s;\bm{0})\notag\\
&=\sum_{\ell=1}^{k}\sum_{\bm{\nu}\in \mathcal{F}_{\ell,s}}\frac{\bm{y}^{\bm{\nu}}}{\bm{\nu}!} \partial^{\bm{\nu}}f(\bm{y}_{s};\bm{0})+\sum_{\bm{\nu}\in \mathcal{F}_{k+1,s}} \frac{k+1}{\bm{\nu}!}\bm{y}^{\bm{\nu}} \int_{0}^{1} (1-t)^k \partial^{\bm{\nu}}f(\bm{y}_s;t\bm{y}_{\{j>s\}}) dt.
\end{align}
where $k$ is such that $C_{k+1}<\infty$. Observe that each term in \eqref{taylor} can be estimated as followings:
\begin{align}
\label{secondext}
&\left\|\sum_{\bm{\nu}\in \mathcal{F}_{k+1,s}} \frac{k+1}{\bm{\nu}!}\bm{y}^{\bm{\nu}} \int_{0}^{1} (1-t)^k \partial^{\bm{\nu}}f(\bm{y}_s;t\bm{y}_{\{j>s\}}) dt\right\|_X\notag\\
&
\leq
 \sum_{\bm{\nu}\in \mathcal{F}_{k+1,s}} \frac{g(\bm{\nu})\bm{\beta}^{\bm{\nu}}}{2^{k+1} \bm{\nu}!}
 \leq
  C_{k+1}
  \sum_{\bm{\nu}\in \mathcal{F}_{k+1,s}} \bm{\beta}^{\bm{\nu}},
\end{align}
and
\begin{align}
\left\| \sum_{\ell=1}^{k}\sum_{\bm{\nu}\in \mathcal{F}_{\ell,s}}\frac{\bm{y}^{\bm{\nu}}}{\bm{\nu}!} \partial^{\bm{\nu}}f(\bm{y}_{s};\bm{0})\right\|_{X} 
\leq
 \sum_{\ell=1}^{k}\sum_{\bm{\nu}\in \mathcal{F}_{\ell,s}}
 \frac{g(\bm{\nu})}{2^\ell \bm{\nu}!} \bm{\beta}^{\bm{\nu}}
 \leq
  C_k\sum_{\ell=1}^{k}\sum_{\bm{\nu}\in \mathcal{F}_{\ell,s}}
  \bm{\beta}^{\bm{\nu}}.
\end{align}
By combining the two results, we obtain
\begin{align}
\label{esti1}
\|f(\bm{y})-f_s(\bm{y}_s)\|_X \leq C\left(\sum_{\bm{\nu}\in \mathcal{F}_{k+1,s}} \bm{\beta}^{\bm{\nu}}+ \sum_{\ell=1}^{k}\sum_{\bm{\nu}\in \mathcal{F}_{\ell,s}}
  \bm{\beta}^{\bm{\nu}} \right). 
\end{align}
Observe that, setting $S_1^k(\beta_j):=\sum_{\ell=1}^{k}\beta_j^\ell$, for all $j>s$,
\begin{align}
\label{S}
\sum_{\ell=1}^{k}\sum_{\bm{\nu}\in \mathcal{F}_{\ell,s}}\bm{\beta}^{\bm{\nu}}
 &\leq \prod_{j>s}\left( 1+ S_1^k(\beta_j)  \right)-1
=\exp\left(\sum_{j>s}\log\left( 1+ S_1^k(\beta_j)\right)  \right)-1\notag\\
&\leq \exp\left(\sum_{j>s}  S_1^k(\beta_j)  \right)-1
\leq   \left( \sum_{j>s} S_1^k(\beta_j) \right)
\exp\left( \sum_{j>s}S_1^k(\beta_j)\right),
\end{align}
where, in the first line, we used our observation that each term on the leftmost side is a term on its right-hand side after the expansion; in the second line, we used the fundamental inequality that $\ln{x}\leq x$ for all $x>0$ with $e^x-1\leq xe^x$. By setting $s>0$ large enough so that $\beta_j<1$ for all $j>s$, using the geometric sum and using $\beta_j$ is decreasing, it is easy to see that $\left(S_1^k(\beta_j)\right)_{j=1}^\infty\in \ell^q$. Indeed, note that
\begin{align}
\label{SS}
S_1^k(\beta_j)=\beta_j\frac{1-\beta_j^k}{1-\beta_j}\leq \frac{\beta_j}{1-\beta_{s+1}}.
\end{align}
Now, use the inequality from Theorem 5.1 of \cite{sch1}, which is
\begin{align}
\label{scheq}
\sum_{j>s} \beta_j\leq \min\left( \frac{r}{1-r},1  \right) \|\bm{\beta}\|_{\ell^r} s^{1-\frac{1}{r}},
\end{align}
for any $\bm{\beta}\in \ell^r$ with  $r\in (0,1)$. Then we obtain 
\begin{align}
\label{second}
\sum_{\ell=1}^{k}\sum_{\bm{\nu}\in \mathcal{F}_{\ell,s}}\bm{\beta}^{\bm{\nu}} \leq C s^{1-\frac{1}{q}}.
\end{align}
 Using \eqref{scheq}, the first term of \eqref{esti1} can be estimated by
\begin{align}
\label{first}
\sum_{\bm{\nu}\in \mathcal{F}_{k+1,s}} \bm{\beta}^{\bm{\nu}}\leq \sum_{\bm{\nu}\in \mathcal{F}_{k+1,s}} \frac{(k+1)!}{\bm{\nu}!}\bm{\beta}^{\bm{\nu}}=\left(\sum_{j>s}\beta_j \right)^{k+1}\leq C s^{(k+1)\left(1-\frac{1}{q} \right)}.
\end{align}
Note that, for any positive integer $k$, it is bounded by $Cs^{1-\frac{1}{q}}$. Thus, by combining \eqref{first} and \eqref{second}, we obtain
\begin{align}
\|f(\bm{y})-f_s(\bm{y}_s)\|_X \leq C s^{1-\frac{1}{q}}.
\end{align}
  Now, letting $\mathcal{G}\in X^*$ be given, we want to estimate a weak truncation error. By the linearity of $\mathcal{G}$, using \eqref{taylor}, observe that
 \begin{align}
 \label{weakrep} 
 \mathbb{E}_{\bm{y}}\left[\mathcal{G}(f(\bm{y}))-\mathcal{G}(f_s(\bm{y}_s))\right]
 &=\sum_{\ell=1}^{k}\sum_{\bm{\nu}\in \mathcal{F}_{\ell,s}}\mathbb{E}_{\bm{y}}\left[ \frac{\bm{y}^{\bm{\nu}}}{\bm{\nu}!} \mathcal{G}(\partial^{\bm{\nu}}f_s(\bm{y}_{s}))\right]
 \notag\\
 &+
 \sum_{\bm{\nu}\in \mathcal{F}_{k+1,s}} \frac{k+1}{\bm{\nu}!}\mathbb{E}_{\bm{y}}\left[\bm{y}^{\bm{\nu}} \int_{0}^{1} (1-t)^k \mathcal{G}(\partial^{\bm{\nu}}f(\bm{y}_s;t\bm{y}_{\{j>s\}})) dt\right].
 \end{align}
We can upper bound the first term in \eqref{weakrep} by
\begin{align}
\label{weakfirst}
\left| \sum_{\ell=1}^{k}\sum_{\substack{\bm{\nu}\in \mathcal{F}_{\ell,s}}}\mathbb{E}_{\bm{y}}\left[ \frac{\bm{y}^{\bm{\nu}}}{\bm{\nu}!} \mathcal{G}(\partial^{\bm{\nu}}f_s(\bm{y}_{s}))\right]\right|
&\leq 
 \sum_{\ell=1}^{k}\sum_{\substack{\bm{\nu}\in \mathcal{F}_{\ell,s}}}\frac{1}{\bm{\nu}!}\left| \mathbb{E}_{\bm{y}}\left[\bm{y}^{\bm{\nu}} \right]\right| \left| \mathbb{E}_{\bm{y}}\left[\mathcal{G}( \partial^{\bm{\nu}}f_s(\bm{y}_{s}))  \right]\right|\notag\\
 &\leq 
 \sum_{\ell=1}^{k}\sum_{\substack{\bm{\nu}\in \mathcal{F}_{\ell,s}\\ \nu_j\neq 1 \text{ for all }j}}\| \mathcal{G}\|_{X^*}\frac{g(\bm{\nu})\bm{\beta}^{\bm{\nu}}}{2^\ell\bm{\nu}!}\notag \\
&\leq  C_k \| \mathcal{G}\|_{X^*} \sum_{\ell=1}^{k}\sum_{\substack{\bm{\nu}\in \mathcal{F}_{\ell,s}\\ \nu_j\neq 1 \text{ for all }j}}\bm{\beta}^{\bm{\nu}}\notag \\
&\leq C'\sum_{\ell=1}^{k}\sum_{\substack{\bm{\nu}\in \mathcal{F}_{\ell,s}\\ \nu_j\neq 1 \text{ for all }j}}\bm{\beta}^{\bm{\nu}},
\end{align}
where the first inequality is by triangular inequality, the second inequality is because $\mathcal{G}$ is bounded and $\mathbb{E}_{\bm{y}}[y_j]=0$ for all $j$. Now, note that 
\begin{align}
\label{S2}
S_2^k(\beta_j):=\sum_{\ell=2}^k \beta_j^\ell=\beta_j^2 \frac{1-\beta_j^{k-1}}{1-\beta_j}\leq \frac{\beta_j^2}{1-\beta_{s+1}},
\end{align}
so we have $(S_2^k(\beta_j))_{j\in \mathbb{N}}\in \ell^{\frac{q}{2}} $. Combining this with \eqref{weakfirst}, the argument in \eqref{S} replacing $S_1^k(\beta_j)$ by $S_2^k(\beta_j)$ and \eqref{scheq}, we obtain
\begin{align}
\label{weak1}
\left| \sum_{\ell=1}^{k}\sum_{\substack{\bm{\nu}\in \mathcal{F}_{\ell,s}}}\mathbb{E}_{\bm{y}}\left[ \frac{\bm{y}^{\bm{\nu}}}{\bm{\nu}!} \mathcal{G}(\partial^{\bm{\nu}}f_s(\bm{y}_{s}))\right]\right|\leq C' s^{1-\frac{2}{q}}.
\end{align}
Using similar argument with  \eqref{secondext} combined with \eqref{first}, we have
 \begin{align}
\label{weak2}
\left|  \sum_{\bm{\nu}\in \mathcal{F}_{k+1,s}} \frac{k+1}{\bm{\nu}!}\mathbb{E}_{\bm{y}}\left[\bm{y}^{\bm{\nu}} \int_{0}^{1} (1-t)^k \mathcal{G}(\partial^{\bm{\nu}}f(\bm{y}_s;t\bm{y}_{\{j>s\}})) dt\right]\right|
&\leq C' s^{(k+1)\left(1-\frac{1}{q} \right)},
\end{align}
where $C'$ absorbed $\|\mathcal{G}\|_{X^*}$ this time. Setting our $k$ to be large enough so that $(k+1)\left(1-\frac{1}{q} \right)\leq 1-\frac{2}{q}$ and combining \eqref{weak1} and \eqref{weak2}, 
\begin{align}
\mathbb{E}_{\bm{y}}\left[\mathcal{G}(f(\bm{y}))-\mathcal{G}(f_s(\bm{y}_s))\right] \leq C' s^{1-\frac{2}{q}}.
\end{align} 
As for the case when $q=1$, it is clear that
\begin{align}
\|f(\bm{y})-f_s(\bm{y}_s)\|_X\leq  C \sum_{j>s} \beta_j
\end{align}
from \eqref{esti1}, \eqref{S}, \eqref{SS} and \eqref{first}. And it is also clear that
\begin{align}
\left| 
\mathbb{E}_{\bm{y}}
\left[
\mathcal{G}(f)-\mathcal{G}(f_s)
\right]
\right|
\leq 
C\left( \sum_{j>s}\beta_j \right)^2,
\end{align}
by setting $k=2$ in \eqref{first} and by  noting from \eqref{S2},
\begin{align}
\sum_{\ell=1}^{k}\sum_{\substack{\bm{\nu}\in \mathcal{F}_{\ell,s}\\ \nu_j\neq 1 \text{ for all }j}}\bm{\beta}^{\bm{\nu}}\leq C' \sum_{j>s} \beta_j^2 \leq C'\left(
\sum_{j>s} \beta_j
 \right)^2.
\end{align}
with varying constant $C'>0$.
\end{proof}
\bigskip

\begin{proof}\textit{(Proof of Lemma \ref{gen})}
Let $\bm{z}$ be constructed as in Theorem 5.10 in \cite{dick}. Then, for any $\theta\in \left(\frac{1}{2},1 \right] $, we have
\begin{align}
\sqrt{\mathbb{E}_{\bm{\Delta}}\left[ \left| \mathbb{E}_{\bm{y}}\left[\mathcal{G}\left(f\right) \right]-Q_{N,s}\mathcal{G}(f)   \right|^2 \right]}
\leq
\left( \frac{1}{\varphi(N)}\sum_{\emptyset \neq \bm{u} \subset\{1:s\}} \gamma_{\bm{u}}^\theta \left(\frac{2\zeta(2\theta)}{(2\pi^2)^\theta }\right)^{|\bm{u}|}  \right)^{\frac{1}{2\theta}}\|\mathcal{G}(f)\|_{s,\gamma}.
\end{align}
The 'unanchored' norm $\|\cdot\|_{s,\gamma}$ is defined in section 4.3 in \cite{dick} as follows:
\begin{align}
\label{qmc1}
\|\mathcal{G}(f)\|_{s,\gamma}^2=\sum_{\bm{u}\subset \{1:s\}}\frac{1}{\gamma_{\bm{u}}} \int _{\left[-\frac{1}{2},\frac{1}{2} \right]^{|\bm{u}|}} \left(
\int_{\left[-\frac{1}{2},\frac{1}{2}  \right]^{s-|\bm{u}|}} \frac{\partial^{|\bm{u}|}}{\partial \bm{y}_{\bm{u}}}\mathcal{G}(f)(\bm{y}) d\bm{y}_{-\bm{u}}
 \right)^2 d\bm{y}_{\bm{u}},
\end{align}
where $\{1:s\}=\{1,2,\dots, s\}$, $-\bm{u}:=\{1:s\}\setminus \bm{u}$, $\bm{y}_{\bm{u}}=(y_j)_{j\in \bm{u}}$ and $\frac{\partial^{|\bm{u}|}}{\partial \bm{y}_{\bm{u}}}:=\prod_{j\in \bm{u}}\frac{\partial}{\partial y_j}$.
\end{proof}
\bigskip

Due to the assumption, we have
\begin{align}
\label{qmc2}
\left| \frac{\partial^{|\bm{u}|}}{\partial \bm{y}_{\bm{u}}} \mathcal{G}(f)(\bm{y}) \right|\leq C\|\mathcal{G}\|_{X^*} (|\bm{u}|!) \bm{\beta}^{\bm{u}}.
\end{align}
Combining \eqref{qmc1} and \eqref{qmc2}, for any $\theta\in \left(\frac{1}{2},1 \right]$, we have
\begin{align}
\label{qmc3}
\sqrt{\mathbb{E}_{\bm{\Delta}}\left[ \left| \mathbb{E}_{\bm{y}}\left[\mathcal{G}\left(f\right) \right]-Q_{N,s}\mathcal{G}(f)   \right|^2 \right]}
\leq
C\|\mathcal{G}\|_{X^*}
C_{\theta,s}\varphi(N)^{-\frac{1}{2\theta}},
\end{align}
where
\begin{align}
C_{\theta,s}:=\left(\sum_{\emptyset \neq \bm{u} \subset\{1:s\}} \gamma_{\bm{u}}^\theta \left(\rho(\theta)\right)^{|\bm{u}|}  \right)^{\frac{1}{2\theta}}\left(\sum_{\bm{u}\subset \{1:s\}} \frac{1}{\gamma_{\bm{u}}} (|\bm{u}|!)^2 \bm{\beta}^{2\bm{u}}  \right)^{\frac{1}{2}},
\end{align}
and
\begin{align}
\rho(\theta):=\left(\frac{2\zeta(2\theta)}{(2\pi^2)^\theta }\right).
\end{align}
Here, $\zeta$ is the Riemann zeta function, and $\varphi$ is the Euler phi function. From \eqref{qmc3}, we can see that if we can find an upper bound for $C_{\theta,s}$ independent on $s$, we can obtain the desired result.  
For this, choose 
\begin{align}
\gamma_{\bm{u}}= \left(\frac{(|\bm{u}|!)^2\bm{\beta}^{2\bm{u}}}{\rho(\theta)^{|\bm{u}|}} \right)^{\frac{1}{1+\theta}},
\end{align}
then we have
\begin{align}
\label{qmcc}
C_{\theta,s}\leq \left( \sum_{\bm{u}\subset \{1:s\}}  (|\bm{u}|!)^{\frac{2\theta}{1+\theta}} \prod_{j\in \bm{u}} \left(\beta_j^{{2\theta}}{\rho(\theta)}\right)^{\frac{1}{1+\theta}}\right)^{\frac{\theta+1}{2\theta}}.
\end{align}
Then it remains to find an upper bound for the sum inside the parenthesis so that the bound is independent of $s$. Observe that
\begin{align}
\label{qmcbound}
\sum_{\bm{u}\subset \{1:s\}}  (|\bm{u}|!)^{\frac{2\theta}{1+\theta}} \prod_{j\in \bm{u}} \left(\beta_j^{{2\theta}}{\rho(\theta)}\right)^{\frac{1}{1+\theta}}
&=
\sum_{\ell=1}^s (|\ell|!)^{\frac{2\theta}{1+\theta}}
\sum_{\substack{\bm{u}\subset \{1:s\}\\|\bm{u}|=\ell}}
\prod_{j\in \bm{u}} \left(\beta_j^{{2\theta}}{\rho(\theta)}\right)^{\frac{1}{1+\theta}}\notag\\
&\leq 
\sum_{\ell=1}^s (|\ell|!)^{\frac{2\theta}{1+\theta}-1}
\left( \sum_{j\geq 1}\left(\beta_j^{{2\theta}}{\rho(\theta)}\right)^{\frac{1}{1+\theta}} \right)^{\ell}.
\end{align}
In order to upper bound the quantity above, we need
\begin{align}
\label{boundfirst}
 \rho(\theta)^{\frac{1}{1+\theta}}\sum_{j\geq 1} \beta_j^{\frac{2\theta}{1+\theta}}=\sum_{j\geq 1}\left(\beta_j^{{2\theta}}{\rho(\theta)}\right)^{\frac{1}{1+\theta}}<\infty,
\end{align}
and
\begin{align}
\label{boundsecond}
\frac{2\theta}{1+\theta}<1\iff \theta<1.
\end{align}
The quantity in \eqref{boundfirst} can be bounded if and only if $\frac{2\theta}{1+\theta}\geq q$ due to the embedding $\ell^r\hookrightarrow \ell^t$ if $r\leq t$ and because $\bm{\beta}\in \ell^q$. This implies that $\theta\geq \frac{q}{2-q}$ and so
\begin{align}
\frac{q}{2-q}\leq \theta <1.
\end{align}
If $q\in \left(0,\frac{2}{3} \right]$, then $\frac{q}{2-q}\leq \frac{1}{2}$ and it implies that $\theta$ could be any number in $\left(\frac{1}{2},1 \right)$ and so $\frac{1}{2\theta}$ could be any number in $\left(\frac{1}{2},1 \right)$. Thus, the error rate is $C \varphi(N)^{-(1-\delta)}$ for arbitrary $\delta\in  \left(0,\frac{1}{2}\right)$. In the case when $q\in \left( \frac{2}{3},1 \right)$, $\theta$ could be any number in $\left[ \frac{q}{2-q},1 \right)$ and so $\frac{1}{2\theta}$ is any number in $\left( \frac{1}{2},\frac{2-q}{2q} \right]$ and the best possible rate is $C\varphi(N)^{-\left( \frac{1}{q}-\frac{1}{2} \right)}$. When $q=1$, let us set $\theta=1$ then observe that
\begin{align}
C_{1,s}\leq \sum_{\bm{u}\subset\{1:s\}} (|\bm{u}|!) \prod_{j\in \bm{u}} \beta_j \rho(1)^{\frac{1}{2}}\leq  \left(\frac{1}{1-\sum_{j\geq 1} \beta_j\rho(1)^{\frac{1}{2}} } \right),
\end{align}
where the second inequality is from Lemma 6.3 in \cite{sch1} provided that $\sum_{j\geq 1} \beta_j <\rho(1)^{-\frac{1}{2}}=\sqrt{6}$.  Lastly, we obtain the desired result for any prime number $N$, noting that $\varphi(N)=N-1$.
\\

\section{An example of the representation \eqref{gprep}}
\label{appendixc}
\\

{
Let us consider the case when $\bm{\nu}=(1,2)=(1,2,0,0,0,\dots)$ with $p=4$ and an analytic function $u$. Observe that, by the usual product rule, the left hand side of \eqref{gprep} is
\begin{align}
\partial^{\bm{\nu}}(u^p) 
&=
 \partial^{(1,2)}(u^4)
\notag\\
& =
 \partial^{(0,2)}(4u^3 
 (
 \partial^{(1,0)}u))
\notag\\
& =
 \partial^{(0,1)}
(12u^2(\partial^{(1,0)}u)(\partial^{(0,1)}u)
+
4u^3 (\partial^{(1,1)}u)
)
\notag\\
&=
24u(\partial^{(0,1)}u)
(\partial^{(1,0)}u)
(\partial^{(0,1)}u)
+
12u^2 
(\partial^{(1,1)}u)
( \partial^{(0,1)}u)
+
12u^2 
(\partial^{(1,0)}u)
( \partial^{(0,2)}u)
\notag\\
&+
12 u^2
(\partial^{(0,1)}u)
( \partial^{(1,1)}u)
+
4u^3 
(\partial^{(1,2)}u)
\notag\\
&=
4u^3 (\partial^{(1,2)}u)
+
12u^2 
(\partial^{(1,0)}u)
( \partial^{(0,2)}u)
+
24 u^2
(\partial^{(0,1)}u)
( \partial^{(1,1)}u)
\notag\\
&+
24u
(\partial^{(0,1)}u)
(\partial^{(1,0)}u)
(\partial^{(0,1)}u).
\end{align}
Now, let us compute the right hand side of \eqref{last} (or equivalently \eqref{gprep}). First, note that $A_{4-0}'$ is empty because there is no $(\bm{m}_i)_{i=1}^{4}$ satisfying $\sum_{i=1}^{4}\bm{m}_i=\bm{\nu}=(1,2)$ and $0<\bm{m}_i<\bm{\nu}$ for all $i=1,2,3,4.$ Next, $A_{4-1}'$ has $3=\frac{3!}{2!1!}$ elements. It is because only possible partition of $(1,2)$ into three components under the two restriction is $(1,0),(0,1)$ and $(0,1)$ and all the possible distinct sequences are $((0,1),(1,0),(0,1))$, $((1,0),(0,1),(0,1))$, $((0,1),(0,1),(1,0))$. Therefore, the coefficient of $u(\partial^{(0,1)}u)(\partial^{(0,1)}u) (\partial^{(1,0)}u)$ is 
\begin{align*}
\frac{4!}{3!1!} \frac{3!}{1!2!} 
\begin{pmatrix}
&(1,2)\\
(1,0)&(0,1)&(0,1)
\end{pmatrix}
=
12\frac{1!}{1!0!0!}\frac{2!}{0!1!1!}=24,
\end{align*}
which is consistent with the left hand side.}\\

{Now, note that $A_{4-2}'$ has two possible partitions of $(1,2)$ in to two components, namely $\{(1,0),(0,2)\}$ and $\{(0,1),(1,1)\}$. And each partition makes two different sequences, $((1,0)(0,2))$ and $((0,2),(1,0))$ for the first partition, and $((0,1)(1,1))$ and $((1,1)(0,1))$ for the second partition. Then the coefficient of $u^2(\partial^{(1,1)}u)(\partial^{(0,1)}u)$ is
\begin{align*}
\frac{4!}{2!2!} \frac{2!}{1!1!} 
\begin{pmatrix}
&(1,2)\\
(1,1)&&(0,1)
\end{pmatrix}
=
12 \frac{1!}{1!0!}\frac{2!}{1!1!}=24
\end{align*}
and, similarly, the coefficient of $u^2 (\partial^{(1,0)}u)
(\partial^{(0,2)}u)$ is
\begin{align*}
\frac{4!}{2!2!}\frac{2!}{1!1!} 
\begin{pmatrix}
&(1,2)\\
(1,0)&&(0,2)
\end{pmatrix}
=12\frac{1!}{1!0!}\frac{2!}{0!2!}=12
\end{align*}
which are consistent with the left hand side again. Lastly, coefficient for $u^3\partial^{(1,2)}u$ is clear from the formula \eqref{gprep}. Therefore, our representation \eqref{gprep} is consistent.}
\\

          \end{document}